\documentclass[12pt]{amsart}
\usepackage{amssymb,amsmath,tabularx,mathrsfs,mathbbol,yfonts,upgreek}
\usepackage{amsthm,verbatim,comment}
\usepackage[bookmarks=true]{hyperref}
\usepackage{pstricks,pst-node,pst-plot}
\usepackage{geometry}
\usepackage{stmaryrd}
\usepackage{paralist}
\usepackage[all]{xy}
\usepackage{array}
\numberwithin{equation}{section}

\DeclareSymbolFontAlphabet{\mathbb}{AMSb}
\DeclareSymbolFontAlphabet{\mathbbl}{bbold}

\newtheorem{thm}{Theorem}[section]
\newtheorem{thmx}{Theorem}

\newtheorem{lem}[thm]{Lemma}
\newtheorem{prop}[thm]{Proposition}
\newtheorem{cor}[thm]{Corollary}

\theoremstyle{definition}

\newtheorem{nota}[thm]{Notation}

\newtheorem{rem}[thm]{Remark}
\theoremstyle{remarks}

\newtheorem*{rem*}{Remarks}

\newtheoremstyle{case}{}{}{}{}{}{:}{ }{}
\theoremstyle{case}

\newcommand{\sym}[1]{\mathfrak{S}_{#1}}

\renewcommand{\O}{\mathcal{O}}

\newcommand{\B}{\mathcal{B}}

\newcommand{\F}{\mathbb{F}}
\newcommand{\K}{\mathbb{K}}

\title[Complexities of Simple Modules of Symmetric groups]{On the complexities of some simple modules of symmetric groups}

\begin{document}
\author{Yu Jiang}
\address[Y. Jiang]{Division of Mathematical Sciences, Nanyang Technological University, SPMS-MAS-05-34, 21 Nanyang Link, Singapore 637371.}
\email[Y. Jiang]{jian0089@e.ntu.edu.sg}


\begin{abstract} Let $p$ be a prime. In this paper, we compute the complexities of some simple modules of symmetric groups labelled by two-part partitions. Most of the simple modules considered here are contained in the $p$-blocks with non-abelian defect groups.
\end{abstract}

\maketitle

\section{introduction}
Let $G$ be a finite group and $\F$ be an algebraically closed field of positive characteristic $p$. The complexity of an $\F G$-module $M$, defined by Alperin and Evens in \cite{JAlperin} and denoted by $c_G(M)$, describes the growth rate of a minimal projective resolution of $M$. However, the invariant is usually very difficult to compute and is still poorly understood even in the context of group algebras of symmetric groups.

Though some advancements have been made in calculating complexities of Young permutation modules, Young modules, and some Specht modules in \cite{DHDN}, \cite{Lim1}, \cite{Lim2} and \cite{Lim3}, little knowledge is known in calculation of complexities of simple modules of symmetric groups. According to our knowledge, the non-trivial cases that were computed mainly focus on the completely splitting representations, the simple modules labelled by hook partitions and some simple modules lying in the $p$-blocks with abelian defect groups (see \cite{DHDN}, \cite{Lim2} and \cite{LimTan}).

The paper contributes to calculation of complexities of simple modules of symmetric groups contained in the $p$-blocks with non-abelian defect groups. Let $C_p$ be a cyclic group of order $p$. Using a theorem about existence of some short exact sequences of $\F C_p$-modules (see Theorem \ref{InsertionRule}), rank varieties of modules and generic Jordan types of modules, we obtain the complexities of some simple modules of symmetric groups labelled by two-part partitions. Following the notations of \cite{GJ1}, let $\sym{n}$ be the symmetric group on $n$ letters and $\lambda$ be a $p$-regular partition of $n$. Denote by $D^\lambda$ the simple $\F \sym{n}$-module labelled by $\lambda$. We will show the following

\begin{thmx}\label{T;A}
Let $n$ be a positive integer.
\begin{enumerate}
\item [\em (i)]Let $\lambda:=(n-1,1)$ be a $p$-regular partition with $p$-weight $w$. Then \[ c_{\sym{n}}(D^{\lambda})=\begin{cases} w-1, &\text{if}\ p=2,\ n\equiv 2\pmod 4,\\
w, & \text{otherwise}.\end{cases}\]
\item [\em (ii)]Let $\lambda:=(n-2,2)$ be a $p$-regular partition with $p$-weight $w$. Then  \[ c_{\sym{n}}(D^{\lambda})=\begin{cases}
w-1, & \text{if}\ p=2,\ n=5\ \text{or}\ 6,\\
w, & \text{otherwise}.\end{cases}\]
\end{enumerate}
\end{thmx}

\begin{thmx}\label{T;C}
Let $m$, $n$, $s$ be three non-negative integers. Let $\lambda:=(n-mp-s,mp+s)$ be a $p$-regular partition with $p$-weight $w$. Then $c_{\sym{n}}(D^{\lambda})=w$ if one of the following hypotheses holds.
\begin{enumerate}
\item [\em (i)]$2<p$, $0\leq m\leq 1$, $0\leq s<p$.
\item [\em (ii)]$3<p$, $m=1$, $s=p$.
\end{enumerate}
\end{thmx}
The organization of the paper is as follows. In Section two, we set up the notations used in the paper and give a brief summary of the required knowledge. Section three presents the proof of Theorem \ref{T;A} and Section four contains the proof of Theorem \ref{T;C}.

\section{preliminaries}
Let $\F$ be an algebraically closed field with positive characteristic $p$ throughout the whole paper. For a given finite group $G$, all $\F G$-modules considered in the paper are finitely generated left $\F G$-modules. Induction and restriction of $\F G$-modules are presented by $\uparrow$ and $\downarrow$. The dual of an $\F G$-module $M$ is denoted by $M^*$. By abusing notations, we also use $\F$ to denote the trivial $\F G$-module throughout the whole paper.

\subsection{Representation theory of finite groups}
We assume that the reader is familiar with modular representation theory of finite groups. For a general background of the topic, one may refer to \cite{JAlperin1} or \cite{HNYT}.

For a given finite group $G$, let $M$ and $N$ be two $\F G$-modules. Write $N\mid M$ if $N$ is isomorphic to a summand of $M$, i.e., $M\cong L\oplus N$ for some $\F G$-module $L$. If $N$ is indecomposable, for a decomposition of $M$ into a direct sum of indecomposable summands, the number of indecomposable summands that are isomorphic to $N$ is well defined by Krull-Schmidt Theorem and is denoted by $[M:N]$. Furthermore, if $N$ is simple, the number of composition factors of $M$ that are isomorphic to $N$ is well defined by Jordan-H\"{o}lder Theorem and is denoted by $(M:N)$. Let $m$ be a positive integer and $N_i$ be an $\F G$-module for all $1\leq i\leq m$. Write $M\sim \sum_{i=1}^m N_i$ to mean that there exists a filtration of $M$ such that all quotient factors are exactly isomorphic to these $N_i$'s. The Jacobson radical of $M$ is denoted by $\text{Rad}(M)$.

Assume further that $M$ is an indecomposable $\F G$-module and $P$ is a subgroup of $G$. Following \cite{JGreen}, say $P$ a vertex of $M$ if $P$ is a minimal (with respect to inclusion of subgroups) subgroup of $G$ subject to the condition that $M\mid M{\downarrow_{P}}{\uparrow^G}$. All vertices of $M$ are known to be $p$-subgroups of $G$ and are $G$-conjugate to each other. Given a vertex $P$ of $M$, there exists some indecomposable $\F P$-module $S$ such that $M\mid S{\uparrow^G}$. It is called a source of $M$. All sources of $M$ are $N_{G}(P)$-conjugate to each other, where $N_{G}(P)$ denotes the normalizer of $P$ in $G$.
\subsection{Rank varieties and complexities of modules}
The $p$-rank of a finite group $G$ is defined to be the largest non-negative integer $m$ subject to the condition that $G$ has an elementary abelian $p$-subgroup of order $p^m$. Let $E$ be an elementary abelian $p$-group of order $p^n$ with a generator set $\{ g_1,\cdots,g_n\}$ and $M$ be an $\F E$-module. For each non-zero point $\alpha$ in the affine space $\F ^n$, one defines $u(\alpha)$ to be $1+\sum_{i=1}^n\alpha_i(g_i-1)$. Note that $\langle u(\alpha)\rangle$ is a cyclic group of order $p$, which means that $M{\downarrow_{\langle u(\alpha)\rangle}}$ is projective if and only if $M{\downarrow_{\langle u(\alpha)\rangle}}$ is free. The rank variety of $M$ over $E$, denoted by $\mathrm{V}_E^{\#}(M)$, is defined to be the set
$$\{0\neq \alpha\in \F^n: M\ \text{is not free as an}\ \F\langle u(\alpha)\rangle\text{-module}\}\cup\{0\}.$$
It is independent of the choice and order of generators of $E$ up to isomorphism. For a fixed ordered generator set of $E$, the $\F E$-modules that are mutually isomorphic have same rank variety. As an algebraic variety, the dimension of $ \mathrm{V}_E^{\#}(M)$ is denoted by $\dim \mathrm{V}_E^{\#}(M)$. Notice that $\dim\mathrm{V}_E^{\#}(M)\leq n$ and $\dim\mathrm{V}_E^{\#}(\F)=n$ by definition. Given a finite group $G$ and an $\F G$-module $M$, if $G$ contains $E$, it is well-known that $\dim\mathrm{V}_E^{\#}(M{\downarrow_E})$ is an invariant up to $G$-conjugation of $E$. The complexity of $M$, denoted by $c_G(M)$, is defined to be the non-negative integer
$$ \max_{E\in \mathcal{E}^{\text{max}}}\{\dim \mathrm{V}_E^{\#}(M{\downarrow_E})\},$$
where $\mathcal{E}^{\text{max}}$ is the set of representatives of all maximal elementary abelian $p$-subgroups of $G$ up to $G$-conjugation. Observe that $c_G(\F)$ is the $p$-rank of $G$. The definition provided here is quite convenient for us in the following discussion. It is also of course equivalent to the classical definition of the complexity of $M$ given in \cite{JAlperin} by \cite[\S 1 Theorem]{JAlperin}, \cite[Theorem 1.1]{AS} and \cite[Theorem 7.6]{Carlson}. If $M$ is an indecomposable $\F G$-module, it is also well-known that $c_G(M)$ is upper bounded by the $p$-rank of defect groups of the $p$-block containing $M$.
Moreover, it attains the upper bound if $p\nmid\dim_\F M$.

\subsection{Generic Jordan types of modules}
Let $C_p$ be a cyclic group of order $p$ and $M$ be an $\F C_p$-module. According to representation theory of cyclic groups, all indecomposable $\F C_p$-modules are exactly the Jordan blocks with all eigenvalues $1$ and size between $1$ to $p$. So $M$ can be viewed as a direct sum of the Jordan blocks. We call the direct sum of the Jordan blocks the Jordan type of $M$ and denote it by $[1]^{n_1}\cdots [p]^{n_p}$, where $[i]$ denotes a Jordan block of size $i$ and $n_i$ means the number of Jordan blocks of size $i$ in the sum. We need the following result about existence of some short exact sequences of $\F C_p$-modules as preparation. It is a special part of \cite[Theorems 4.1, 4.2]{Klein}. For our purpose, we present a version that is convenient for us to compute.
\begin{thm}\label{InsertionRule}
Let $C_p$ be a cyclic group of order $p$. Let $U$, $V$ and $W$ be three $\F C_p$-modules, where $U$ has Jordan type $[a]$ and $V$, $W$ have Jordan types $[1]^{m_1}\cdots[p]^{m_p}$, $[1]^{n_1}\cdots[p]^{n_p}$ respectively. Then a short exact sequence
$$ 0\rightarrow U\rightarrow V\rightarrow W\rightarrow 0$$
exists if and only if one of the following situations holds.
\begin{enumerate}
\item [\em (i)]$m_a\neq 0$ and we have \[n_i=\begin{cases} m_i-1, &\text{if}\ i=a,\\
m_i, & \text{otherwise}.\end{cases}\]
\item [\em (ii)]There exist some positive integer $c$ such that $a<c\leq p$ and $m_c\neq 0$ and a set of mutually distinct positive integer $3$-tuples (may be empty) $$\bigcup_{u=1}^t\{(i_u,i_u-q_u+1,\ell_u)\},$$ where
    \begin{align*}
    & 1\leq q_t\leq\cdots \leq q_1<a,\ 0\leq i_t-q_t\leq\cdots\leq i_1-q_1\leq c-a,\\
    &1\leq i_t\leq\cdots\leq i_1<c\ \text{and}\ 1\leq \ell_u\leq m_{i_u}.
    \end{align*}
    In the situation, we have
    \[n_i=\begin{cases} m_i-\delta_{i,c}-\sum_{u=1}^t\delta_{i,i_u}+\delta_{i,c-a+q_1}+\sum_{u=1}^{t-1}\delta_{i,i_u-q_u+q_{u+1}}+
    \delta_{i,i_t-q_t}, &\text{if}\ 1\leq i\leq c,\\
    m_i & \text{otherwise},\end{cases}\]
    where $\delta_{x,y}$ is the usual Kronecker delta for two given integers $x$ and $y$. If the set of $3$-tuples is empty, we have $n_c=m_c-1$, $n_{c-a}=m_{c-a}+1$ and $n_i=m_i$ if $i\neq c,\ c-a$.
\end{enumerate}
\end{thm}
For a general case where the $\F C_p$-module $U$ in Theorem \ref{InsertionRule} is decomposable, if a short exact sequence in Theorem \ref{InsertionRule} exists, one can use Theorem \ref{InsertionRule} and Third Isomorphism Theorem to determine all allowable Jordan types of $W$.

Now take $E$ to be an elementary abelian $p$-group of order $p^n$ with a generator set $\{g_1,\cdots,g_n\}$ and $M$ to be an $\F E$-module. Let $\K$ denote the extension field $\F(\alpha_1,\cdots,\alpha_n)$ over $\F$ where $\{\alpha_1,\cdots,\alpha_n\}$ is a set of indeterminates. Denote by $u_{\alpha}$ the element $1+\sum_{i=1}^n\alpha_i(g_i-1)$ of $\K E$. Note that $\langle u_\alpha\rangle$ is a cyclic group of order $p$ and $M$ can be viewed as a $\K \langle u_\alpha\rangle$-module naturally. Therefore, Jordan type of $M{\downarrow_{\langle u_\alpha\rangle}}$ is meaningful and is called generic Jordan type of $M$. Wheeler in \cite{WW} showed that generic Jordan type of $M$ is independent of the choice of generators of $E$. The stable generic Jordan type of $M$ is the generic Jordan type of $M$ modulo free direct summands. Let $M$ have generic Jordan type $[1]^{n_1}\cdots [p]^{n_p}$. Module $M$ is said to be generically free if $n_1=\cdots=n_{p-1}=0$. Finally, for an $\F E$-module $N$, say $N$ and $M$ have mutually complementary stable generic Jordan types if $N$ has stable generic Jordan type $[1]^{n_{p-1}}\cdots [p-1]^{n_1}$.

Notice that the modules isomorphic to $M$ as $\F E$-modules and $M$ have same generic Jordan type. Given a finite group $G$ containing $E$ and an $\F G$ module $N$, for any $g\in G$, let $E^g:=g^{-1}Eg$ and note that $N{\downarrow_{E^g}}$ and $N{\downarrow_E}$ have same generic Jordan type. For more properties of generic Jordan types of modules, one may refer to \cite{EFJPAS} and \cite{WW}. Some needed results of generic Jordan type of modules are collected as follows.
\begin{prop}\label{P;sjt}
Let $E$ be an elementary abelian $p$-subgroup of a finite group $G$. Let $M$ be an $\F E$-module and $N$ be an indecomposable $\F G$-module.
\begin{enumerate}
\item [\em (i)]Module $M$ is generically free only if $p\mid \dim_\F M$.
\item [\em (ii)] Modules $M$ and $M^*$ have same stable generic Jordan type.
\item [\em (iii)]The stable generic Jordan type of a direct sum of $\F E$-modules is the direct sum of the stable generic Jordan types of these modules.
\item [\em (iv)]Let $0\rightarrow M_1\rightarrow M_2\rightarrow M_3\rightarrow 0$ be a short exact sequence of $\F E$-modules.
    \begin{enumerate}
    \item [\em (a)]If $M_1$ is generically free, then $M_2$ and $M_3$ have same stable generic Jordan type.
    \item [\em (b)]If $M_2$ is generically free, then $M_1$ and $M_3$ have mutually complementary stable generic Jordan types.
    \item [\em (c)]If $M_3$ is generically free, then $M_1$ and $M_2$ have same stable generic Jordan type.
    \end{enumerate}
\item [\em (v)]Let $F$ be a proper subgroup of $E$ and $L$ be an $\F F$-module. Then $L{\uparrow^E}$ is generically free.
\item [\em (vi)]Module $M$ is not generically free if and only if $\mathrm{V}_E^{\#}(M)=\mathrm{V}_E^{\#}(\F)$.
\item [\em (vii)]Let $E$ have $p$-rank $n$. If $N{\downarrow_E}$ is not generically free, then $n\leq c_G(N)$. In the case, $c_G(N)=n$ if $n$ is also the $p$-rank of defect groups of the $p$-block containing $M$.
\end{enumerate}
\end{prop}
Let $G$ be a finite group and $M$ be an $\F G$-module. Module $M$ is said to be a $p$-permutation module if, for every $p$-subgroup $P$ of $G$, there exists a basis $\B$ of $M$ such that it can be permuted by $P$. Note that the class of $p$-permutation modules is closed by taking direct sum and direct summands. It is well-known that indecomposable $p$-permutation modules have trivial sources. Let $E$ be an elementary abelian $p$-subgroup of $G$, if M is also a $p$-permutation module, the stable generic Jordan type of $M{\downarrow_E}$ is described as follows.
\begin{lem}\cite[Lemma 3.1]{JLW}\label{T;permutaiton}
Let $E$ be an elementary abelian $p$-subgroup of a finite group $G$. Let $M$ be a $p$-permutation $\F G$-module and $\B$ be a basis of it permuted by $E$. The stable generic Jordan type of $M{\downarrow_E}$ is $[1]^n,$ where $n$ is the number of fixed points of $\B$ under the action of $E$.
\end{lem}
\subsection{Combinatorics}
Let $n$ be a positive integer and $[n]$ be the set $\{1,\cdots,n\}$. Let $\sym{n}$ be the symmetric group acting on $[n]$. Given positive integers $m_1,\cdots, m_\ell$, we view $\sym{m_1}\times\cdots\times\sym{m_\ell}$ as a subgroup of $\sym{n}$ naturally if $\sum_{i=1}^{\ell}m_i\leq n$. A partition of a non-negative integer $m$ is a non-increasing sequence of non-negative integers $(\lambda_1,\cdots,\lambda_{\ell})$ such that $\sum_{i=1}^\ell\lambda_i=m$. By abusing notations, the unique partition of $0$ is denoted by $\varnothing$. As usual, let $\unlhd$ be the dominance order on the set of all partitions of $n$. Given a partition $\lambda=(\lambda_1,\cdots,\lambda_{\ell})$, denote by $|\lambda|$ the sum $\sum_{i=1}^\ell \lambda_i$ and denote by $[\lambda]$ the set
$$ \{(i,j)\in \mathbb{N}^2:\ 1\leq i\leq \ell,\ 1\leq j\leq \lambda_i\}.$$
The set is called the Young diagram of $\lambda$. We will not distinguish between $\lambda$ and $[\lambda]$. The $p$-core of $\lambda$, denoted by $\kappa_p(\lambda)$, is the partition whose Young diagram is constructed by removing all rim $p$-hooks from $\lambda$ successively. It is well-known that the $p$-cores of partitions of $n$ label the $p$-blocks of $\F \sym{n}$. The number of rim $p$-hooks removed from $\lambda$ to obtain $\kappa_p(\lambda)$ is called the $p$-weight of $\lambda$. It is equal to $\frac{|\lambda|-|\kappa_p(\lambda)|}{p}$. We say $\lambda$ is $p$-regular if there does not exist a non-negative integer $m$ such that $m+p\leq \ell$ and $\lambda_{m+1}=\cdots=\lambda_{m+p}$.
The set of all $p$-regular partitions of $n$ is denoted by $\mathfrak{Reg}_p(n)$. The partition $\lambda$ is said to be $p$-restricted if the conjugate of it is $p$-regular. The $p$-adic expansion of $\lambda$ is the unique sum $\sum_{i=0}^mp^i\lambda(i)$ for some non-negative integer $m$, where $\lambda(i)$ is $p$-restricted for all $0\leq i\leq m$ and $\lambda=\sum_{i=0}^mp^i\lambda(i)$ via component-wise addition and scalar multiplication of partitions.
\subsection{Modules of symmetric groups}
We now briefly present some materials of representation theory of symmetric groups needed in the paper. One can refer to \cite{GJ1} or \cite{GJ3} for a background of the topic. For a given subgroup $H$ of $\sym{n}$, denote by $[n]/H$ the set of orbits of $[n]$ under the natural action of $H$. Let $a$ be an integer satisfying $0\leq ap\leq n$. If $0<a$, let $E_a$ be the elementary abelian $p$-subgroup $\langle \cup_{i=1}^a\{s_i\}\rangle$ of $\sym{n}$ where $s_i:=((i-1)p+1,(i-1)p+2,\cdots, ip).$ Set $E_0$ to be the trivial group and note that $E_a$ has $p$-rank $a$. Let $b$ and $c$ be two positive integers such that $b+c\leq n$ and let $B$ and $C$ be subgroups of $\sym{b}$ and $\sym{c}$ respectively. We write $B\times C\hookrightarrow \sym{b}\times \sym{c}\hookrightarrow \sym{n}$ for the obvious inclusions of groups. Let $\lambda:=(\lambda_1,\cdots,\lambda_\ell)$ be a partition of $n$. The Young subgroup associated with $\lambda$, denoted by $\mathfrak{S}_\lambda$, is defined to be
$$ \sym{\lambda_1}\times\cdots\times\sym{\lambda_\ell},$$
where the first factor $\sym{\lambda_1}$ acts on the set $\{1,\cdots, \lambda_1\}$, the second factor acts on the set $\{\lambda_1+1,\cdots, \lambda_1+\lambda_2\}$ and so on. To describe modules of symmetric groups, we assume that the reader is familiar with the definitions of tableau, tabloid and polytabloid (see \cite{GJ1} or \cite{GJ3}). Let $\F(\lambda)$ be the trivial $\F\sym{\lambda}$ module. The Young permutation module with respect to $\lambda$, denoted by $M^{\lambda}$, is defined to be the induced module $\F(\lambda){\uparrow^{\sym{n}}}$. It is an $\F$-space generated by all $\lambda$-tabloids where $\sym{n}$ permutes these tabloids. Notice that $M^\lambda$ is a $p$-permutation module since all $\lambda$-tabloids form a basis that can be permuted by any $p$-subgroup of $\sym{n}$. Let $E$ be an elementary abelian $p$-subgroup of $\sym{n}$. We get stable generic Jordan type of $M^\lambda{\downarrow_E}$ by Lemma \ref{T;permutaiton}.
\begin{lem}\label{PModules}\cite[Lemma 2.7]{JLW}
Let $\lambda$ be a partition of $n$ and $E$ be an elementary abelian $p$-subgroup of $\sym{n}$. Then $M^\lambda{\downarrow_E}$ has stable generic Jordan type $[1]^m,$ where $m$ is the number of unordered ways to insert the orbits in $[n]/E$ into the rows of $\lambda$.
\end{lem}

The Specht module $S^\lambda$ is an $\F\sym{n}$-submodule of $M^{\lambda}$ generated by all $\lambda$-polytabloids. It has a basis labelled by all standard $\lambda$-tableaux by Standard Basis Theorem. The dual of it is denoted by $S_\lambda$. Unlike the characteristic zero case, some Specht modules are no longer simple now. However, James in \cite[Theorem 11.5]{GJ1} showed that the module $D^{\lambda}:=S^{\lambda}/\text{Rad}(S^{\lambda})$ is either zero or simple. It is simple if and only if $\lambda$ is $p$-regular and the set $ \{D^{\mu}: \mu\in \mathfrak{Reg}_p(n)\}$
forms a complete set of representatives of isomorphic classes of simple $\F\sym{n}$-modules. For any $\mu\in \mathfrak{Reg}_p(n)$, it is well-known that the $p$-rank of defect groups of the $p$-block containing $D^\mu$ is exactly the $p$-weight of $\mu$. Moreover, $D^\mu$ is self dual and $(S^\lambda:D^\mu)\neq 0$ only if $\lambda\trianglelefteq \mu$.

The representatives of all isomorphism classes of indecomposable summands of Young permutation modules are called Young modules and are parametrized by James in \cite[Theorem 3.1]{GJ2} using partitions. Namely, let $\lambda$ and $\mu$ be two partitions of $n$. For a given decomposition of $M^\lambda$ into a direct sum of indecomposable summands, there exists a unique indecomposable summand of $M^\lambda$ that contains $S^\lambda$. The module, unique up to isomorphism, is denoted by $Y^\lambda$. We have $[M^\lambda:Y^\lambda]=1$ and $[M^{\lambda}:Y^{\mu}]\neq 0$ only if $\lambda\unlhd \mu$. One thus gets
\begin{equation}
M^\lambda\cong Y^\lambda\oplus\bigoplus_{\lambda\lhd \mu}k_{\lambda,\mu}Y^\mu,
\end{equation}
where $k_{\lambda,\mu}:=[M^\lambda:Y^\mu]$ and all $k_{\lambda,\mu}$'s are known as $p$-Kostka numbers. Note that Young modules are self-dual and are $p$-permutation modules with trivial sources. It is well-known that $Y^\lambda$ has a filtration for some non-negative integer $m$
\begin{equation}
0=M_0\subset M_1\subset\cdots\subset M_{m+1}=Y^\lambda,
\end{equation}
where $M_{i+1}/M_{i}$ is isomorphic to some Specht module $S^{\lambda_i}$ and $\lambda\unlhd\lambda_i$ if $0\leq i\leq m$. Moreover, $M_1\cong S^\lambda$ and the quotient factor $S^\lambda$ occurs exactly once in the filtration of $Y^\lambda$.
The vertices of Young modules are determined by Grabmeier in \cite{GJ} and by Erdmann in \cite{Erdmann} using only representation theory of symmetric groups. To describe them, let $\lambda$ have $p$-adic expansion $\sum_{i=0}^mp^i\lambda(i)$ for some non-negative integer $m$, where $|\lambda(i)|=n_i$ for all $0\leq i\leq m$. Denote by $\O_\lambda$ the partition
$$(\underbrace{p^m,\cdots,p^m}_{n_m\ \text{times}},\ \underbrace{p^{m-1},\cdots,p^{m-1}}_{n_{m-1}\ \text{times}},
\cdots,\ \underbrace{1,\cdots,1}_{n_0\ \text{times}}).$$

\begin{thm}\label{T; Youngvertices}(\cite{Erdmann},\ \cite{GJ})
Let $\lambda$ be a partition of $n$. Then the vertices of $Y^{\lambda}$ are $\sym{n}$-conjugate to the Sylow $p$-subgroups of $\mathfrak{S}_{\O_\lambda}$.
\end{thm}

Some known results used in the paper are collected as follows.
\begin{lem}\label{L;YongModules}
Let $\lambda\in \mathfrak{Reg}_p(n)$ having $p$-weight $w$. Let $E$ be an elementary abelian $p$-subgroup of $\sym{n}$ and $P_\lambda$ be a Sylow $p$-subgroup of $\mathfrak{S}_{\O_\lambda}$. Then
\begin{enumerate}
\item [\em (i)]\cite[Proposition 1.1]{DH} $S^\lambda\cong D^\lambda\cong Y^\lambda$ if $S^\lambda\cong D^\lambda$.
\item [\em (ii)]\cite[Theorem 3.5]{JLW} $Y^\lambda{\downarrow_E}$ has stable generic Jordan type $[1]^a$ where $0\leq a$. It is not generically free if and only if $E$ is $\sym{n}$-conjugate to a $p$-subgroup of $P_\lambda$.
\item [\em (iii)]\cite[Corollary 6.5]{DanzLim} $c_{\sym{n}}(D^{\lambda})=w$ if $2<p$ and $S^\lambda\cong D^\lambda$.
\end{enumerate}
\end{lem}

We now specialize our discussion to the modules labelled by two-part partitions. A two-part partition of $n$ is a partition of $n$ with at most two parts. We begin with some notations. For two non-negative integers $a$ and $b$, let $a=\sum_{0\leq i}p^ia_i$ and $b=\sum_{0\leq i}p^ib_i$ where $0\leq a_i,b_i<p$. The sums are uniquely written and known as the $p$-adic sums of $a$ and $b$. Write $a\subseteq_p b$ if $0\leq a_i\leq b_i$ for all $0\leq i$ and write $a \sqsubseteq_p b$ if $a_i=b_i$ or $a_i=0$ for all $0\leq i$.

The following results are necessary in the discussion.

\begin{thm}\label{T;Henke}\cite[Theorem 3.3]{Henke}
Let $s$ and $k$ be two integers such that $0\leq 2s\leq 2k\leq n$. Then $[M^{(n-k,k)}:Y^{(n-s,s)}]=\varphi(n,k,s)$,
where
\[
\varphi(n,k,s):=\begin{cases} 1, & \text{if}\ k-s\subseteq_p n-2s,\\
0, & \text{otherwise}.\end{cases}
\]
\end{thm}

\begin{thm}\label{T;James}\cite[Theorem 24.15]{GJ1}
Let $s$ and $k$ be two integers such that $0\leq 2s\leq 2k\leq n$ and $(n-s,s)\in \mathfrak{Reg}_p(n)$. Then $(S^{(n-k,k)}:D^{(n-s,s)})=\psi(n,k,s)$,
where
\[
\psi(n,k,s):=\begin{cases} 1, & \text{if}\ k-s\sqsubseteq_p n-2s+1,\\
0, & \text{otherwise}.\end{cases}
\]
\end{thm}
According to \cite[Theorem 17.13]{GJ1}, for a partition $(n-m,m)$, we remark that $M^{(n-m,m)}$ has a filtration
\begin{equation}
0=M_0\subset M_1\subset\cdots\subset M_{m+1}=M^{(n-m,m)},
\end{equation}
where $M_{i+1}/M_{i}\cong S^{(n-m+i,m-i)}$ if $0\leq i\leq m$. In particular, one can deduce that
\begin{equation}
\dim_\F S^{(n-m,m)}={n\choose m}-{n\choose m-1}.
\end{equation}

The following lemma will be used in the proof of Theorem \ref{T;C}.

\begin{lem}\label{R;Spechtfiltration}
Let $\lambda$ and $\mu$ be two $p$-regular two-part partitions of $n$. They satisfy the following conditions:
\begin{enumerate}
\item [\em (a)]$\lambda\lhd\mu\lhd(n)$, $\kappa_p(\lambda)=\kappa_p(\mu)=\kappa$;
\item [\em (b)]there does not exist a partition $\nu$ such that $\lambda\lhd\nu\lhd\mu$ and $\kappa_p(\nu)=\kappa$;
\item [\em (c)]there does not exist a partition $\eta$ such that $\mu\lhd\eta\lhd(n)$ and $\kappa_p(\eta)=\kappa$.
\end{enumerate}
\begin{enumerate}
\item [\em (i)] Let $\kappa_p((n))=\kappa$, $Y^{(n)}\mid Y^\lambda$ and $Y^\mu\nmid M^\lambda$. Then there exists a short exact sequence $0\rightarrow S^{\lambda}\rightarrow Y^\lambda\rightarrow S^{\mu}\rightarrow 0$.
\item [\em (ii)] Let $\kappa_p((n))\neq\kappa$ and $Y^\mu\nmid M^\lambda$. Then there exists a short exact sequence $0\rightarrow S^{\lambda}\rightarrow Y^\lambda\rightarrow S^{\mu}\rightarrow 0$.
\end{enumerate}
\end{lem}
\begin{proof}
Let $e_\kappa$ be the block idempotent of the $p$-block of $\F\sym{n}$ labelled by $\kappa$ and $M:=e_\kappa M^\lambda$. For (i), as $\kappa_p((n))=\kappa$, note that $M\sim S^\lambda+S^\mu+S^{(n)}$ by $(2.3)$ and the conditions. We also get $M\cong Y^\lambda\oplus Y^{(n)}$ from $(2.1)$, Theorem \ref{T;Henke}, Krull-Schmidt Theorem and the hypotheses of (i). Now consider about the filtration of $Y^\lambda$ in $(2.2)$ and suppose that $S^\mu$ is not a quotient factor of the filtration. By the conditions again, note that all the quotient factors of the filtration are the trivial modules except one $S^\lambda$. It contradicts with the equation $(M:D^\mu)=(Y^\lambda:D^\mu)$ by the condition (a). So $S^\mu$ is a quotient factor of the filtration. Then (i) follows by $(2.2)$, the condition (a) and counting dimensions of $M$ and $Y^\lambda$.

For (ii), by same reasons given in (i) and the hypotheses of (ii), we can deduce $M\sim S^\lambda+S^\mu$, $M\cong Y^\lambda$. Moreover, by $(2.3)$ and the condition (a), the quotient factor $S^\mu$ in the filtration of $M$ occurs on the top. The proof of (ii) is complete.
\end{proof}

We now turn to discuss restriction of the simple $\F\sym{n}$-modules labelled by two-part partitions. Restriction of these modules has been studied by Sheth in \cite{Sheth}. For a general version of induction and restriction of simple $\F\sym{n}$-modules, which is referred as modular branching rule, one can refer to
\cite[Theorems 11.2.7, 11.2.8]{A.K}.

The following result will be repeatedly used in the paper.
\begin{thm}\label{T;MBR}\cite[Theorem (ii), (iii)]{Sheth}
Let $2<p$ and $a$, $b$ be two integers. When $0<b<a$, we have
\[ D^{(a,b)}{\downarrow_{\sym{a+b-1}}}\cong\begin{cases} D^{(a-1,b)}\oplus D^{(a,b-1)}, &\text{if}\ b-2\not\equiv a-1,\ a\pmod p,\\
D^{(a-1,b)}, &\text{if}\ b-2\equiv a \ \ \ \ \ \ \ \ \ \pmod p.\end{cases}\]
\end{thm}

\section{Proof of Theorem \ref{T;A}}
This section is designed to prove Theorem \ref{T;A}. Throughout the whole section, for the partitions $(n-1,1)$ and $(n-2,2)$, let $S_1:=S^{(n-1,1)}$ and $S_2:=S^{(n-2,2)}$. Moreover, let $D_1:=D^{(n-1,1)}$ and $D_2:=D^{(n-2,2)}$ if these partitions are $p$-regular. For any elementary abelian $p$-group $E$, notice that the trivial $\F E$-module always has stable generic Jordan type $[1]^1$.
\begin{nota}\label{N;notation3}
We fix some notations as follows.
\begin{enumerate}[(i)]
\item Let $P$ be a $p$-subgroup of $\sym{n}$. When $P$ acts on $[n]$ naturally, write $a_0(P)$ and $a_1(P)$ to denote the number of fixed points and the number of orbits of size $p$ respectively.
\item Let $2\leq n$ and $\{t_i\}$ be the $(n-1,1)$-tabloid with $i$ lying in the second row.  When $p=2$ and $2\mid n$, $D_1$ has a basis $\{\{t_1\}+\{t_i\}+\text{Rad}(S_1):\ 2\leq i\leq n-1\}$, where $\text{Rad}(S_1)$ is generated by $\sum_{i=1}^n\{t_i\}$. The basis is denoted by $\mathcal{B}_2^n$.
\item Let $4\leq n$ and $\{t_{i,j}\}$ be the $(n-2,2)$-tabloid with $i$, $j$ lying in the second row. Denote by $M_{0}^{2}(n)$ the set $$\{\sum_{1\leq i<j\leq n}k_{i,j}\{t_{i,j}\}\in M^{(n-2,2)}:\ \sum_{1\leq i<j\leq n}k_{i,j}=0\}.$$ Note that $M_{0}^{2}(n)$ is an $\F\sym{n}$-submodule of $M^{(n-2,2)}$. Also notice that $M_{0}^{2}(n)$ has a basis $\{\{t_{i,j}\}-\{t_{1,2}\}:\ 1\leq i<j \leq n,\ \{i,j\}\neq \{1,2\}\}$ which is denoted by $\B_{2,0}^n$. Moreover, there exists a short exact sequence
    \begin{equation}
    0\rightarrow M_{0}^{2}(n)\rightarrow M^{(n-2,2)}\xrightarrow{\delta}\F\rightarrow 0,
    \end{equation}
    where $\delta$ is the augmentation map from $M^{(n-2,2)}$ to $\F$.
\item Let $C_p$ be a cyclic group of order $p$. For a positive integer $m$, there exists a left regular action from $(C_p)^m$ to itself which is faithful. This action induces an injective group homomorphism from $(C_p)^m$ to $\sym{p^m}$. The homomorphic image is denoted by $R_{m,p}$.
\item Let $a$ be an integer where $0\leq 4a\leq n$. If $0<a$, let $K_a$ be the elementary abelian $2$-subgroup $\langle \cup_{i=1}^a\{k_{i,1},k_{i,2}\}\rangle$ of $\sym{n}$ where $k_{i,1}:=(4i-3,4i-2)(4i-1,4i)$ and $k_{i,2}:=(4i-3,4i-1)(4i-2,4i).$ Set $K_0$ to be the trivial group and note that $K_a$ has $2$-rank $2a$.
\item Let $2\mid n$ and $\ell$, $m$ be two integers such that $0\leq 4\ell\leq n$, $m=\frac{n-4\ell}{2}$. If $\ell<\frac{n}{4}$, let $F_\ell$ be the $2$-subgroup $\langle\cup_{i=1}^{m}\{s_{\ell,i}\}\rangle$ of $\sym{n}$ where $s_{\ell,i}:=(4\ell+2i-1,4\ell+2i)$. Set $F_{n/4}$ to be the trivial group if $4\mid n$ and note that the $2$-rank of $F_\ell$ is $m$. Moreover, $F_0=E_{n/2}$ if $p=2$.
\end{enumerate}
\end{nota}
We need the following lemmas as preparation.

\begin{lem}\label{L;(n-1,1)sjt}
Let $E$ be an elementary abelian $p$-subgroup of $\sym{n}$. Let $a_0$ and $a_1$ be $a_0(E)$ and $a_1(E)$ respectively. Then $S_1{\downarrow_E}$ has stable generic Jordan type \[\begin{cases} [p-1]^1, &\text{if}\ a_{0}=0,\\
[1]^{a_0-1}, &\text{if}\ 0<a_0.
\end{cases}\]
\end{lem}
\begin{proof}
When $a_0=0$, observe that $p\mid n$. In the case, we have $M^{(n-1,1)}\cong Y^{(n-1,1)}$ by Theorem \ref{T;Henke} and $Y^{(n-1,1)}{\downarrow_E}$ is generically free by Lemma \ref{PModules}. We also have a short exact sequence $0\rightarrow S_1{\downarrow_E}\rightarrow Y^{(n-1,1)}{\downarrow_E}\rightarrow S^{(n)}{\downarrow_E}\rightarrow 0.$ The stable generic Jordan type of $S_1{\downarrow_E}$ is $[p-1]^1$ by using Proposition \ref{P;sjt} (iv) (b) to above short exact sequence. For the left case $a_0\neq 0$, we may assume that $\{t_1\}$ is fixed by $E$. Then we note that $S_1{\downarrow_E}$ is itself a $p$-permutation $\F E$-module since the standard basis of $S_1$ can be permuted by $E$ and its subgroups. The result in the case now follows by Lemma \ref{T;permutaiton}.
\end{proof}

\begin{lem}\label{L;sjt(n-2,2)1}
Let $4\leq n$ and $E$ be an elementary abelian $p$-subgroup of $\sym{n}$. Let $a_0$ and $a_1$ be $a_0(E)$ and $a_1(E)$ respectively. Then $M_{0}^2(n){\downarrow_E}$ has stable generic Jordan type \[\begin{cases} [p-1]^1, & \text{if}\ 2<p,\ a_0<2\ \text{or}\ p=2,\ a_0<2,\ a_1=0,\\
[1]^{{a_0\choose 2}-1}, &\text{if}\ 2<p,\ 2\leq a_0,\\
[1]^{{a_0\choose 2}+a_1-1}, &\text{otherwise}.
\end{cases}\]
\end{lem}
\begin{proof}
When $2<p$, $a_0<2$ or $p=2$, $a_0<2$, $a_1=0$, observe that $M^{(n-2,2)}{\downarrow_E}$ is generically free by Lemma \ref{PModules}. The fact implies that $M_{0}^2(n){\downarrow_E}$ has stable generic Jordan type $[p-1]^1$ by applying Proposition \ref{P;sjt} (iv) (b) to the short exact sequence in $(3.1)$. When $2<p$ and $2\leq a_0$, we may assume that $\{t_{1,2}\}$ is fixed by $E$. Note that $M_{0}^2(n){\downarrow_E}$ is a $p$-permutation $\F E$-module since $\B_{2,0}^n$ can be permuted by $E$ and its subgroups. We get the desired stable generic Jordan type of $M_{0}^2(n){\downarrow_E}$ by using Lemma \ref{T;permutaiton}. In all other cases, namely $p=2$, $2\leq a_0$  or $p=2$, $a_0<2$, $0<a_1$, we may also assume that $\{t_{1,2}\}$ is fixed by $E$.  So $M_{0}^2(n){\downarrow_E}$ is still a $p$-permutation $\F E$-module as $\B_{2,0}^n$ is also permuted by $E$ and its subgroups. We thus get the desired stable generic Jordan type of $M_{0}^2(n){\downarrow_E}$ by using Lemma \ref{T;permutaiton} again.
\end{proof}

\begin{lem}\label{L;sjt(n-2,2)2}
Let $p\nmid n$ and $E$ be an elementary abelian $p$-subgroup of $\sym{n}$. Let $a_0$ and $a_1$ be $a_0(E)$ and $a_1(E)$ respectively. Then $S_2{\downarrow_E}$ has stable generic Jordan type \[\begin{cases} [p-1]^1, & \text{if}\ 2<p,\ a_0<3\ \text{or}\ p=2,\ a_0<3,\ a_1=0,\\
[1]^{{a_0-1\choose 2}-1}, &\text{if}\ 2<p,\ 3\leq a_0,\\
[1]^{{a_0-1\choose 2}+a_1-1}, &\text{otherwise}.
\end{cases}\]
\end{lem}
\begin{proof}
When $p\nmid n$, note that $E$ can be viewed as an elementary abelian $p$-subgroup of $\sym{n-1}$. In the case $4<n$, we have $S_2{\downarrow_E}\cong S_2{\downarrow_{\sym{n-1}}}{\downarrow_E}\cong M_{0}^2(n-1){\downarrow_E}$ according to \cite[Lemma 3]{Peel}. When $n=4$, we have $S^{(2,2)}{\downarrow_E}\cong S^{(2,1)}{\downarrow_E}$ by Branching Rule. The proof is now complete by Lemmas \ref{L;sjt(n-2,2)1} and \ref{L;(n-1,1)sjt}.
\end{proof}

\begin{lem}\label{L;p|n(n-2,2)}
Let $2<p$ and $p\mid n$. Let $E$ be an elementary abelian $p$-subgroup of $\sym{n}$ and $a_0$ be $a_0(E)$. Then $S_2{\downarrow_E}$ has stable generic Jordan type $[1]^{{a_0\choose 2}-a_0}$.
\end{lem}
\begin{proof}
When $2<p$ and $p\mid n$, we have $p\nmid n-2$. By \cite[Theorems 1, 2]{Peel}, it means $M_{0}^2(n)\cong S_1\oplus S_2$. Note that $a_0\neq 1$, $2$ in the case as $p\mid n$. The lemma follows by Lemmas \ref{L;(n-1,1)sjt}, \ref{L;sjt(n-2,2)1} and Proposition \ref{P;sjt} (iii).
\end{proof}

For our purpose, one more lemma is needed for the case $p=2$.

\begin{lem}\label{L;p=2,sjt}\cite[Lemma 5.9]{Lim1}
Let $p=2$ and $2\mid n$. Then $S_2{\downarrow_{E_{n/2}}}$ has stable generic Jordan type
$[1]^{\frac{n}{2}-2}$.
\end{lem}

We state a result to figure out all maximal elementary abelian $p$-subgroups of $\sym{n}$.

\begin{thm}\label{T;Maximal subgroup}\cite[$\S$ VI Theorem 1.3]{AM}
Conjugate representatives of all maximal elementary abelian $p$-subgroups of $\sym{n}$ are
\begin{align*}
 &\underbrace{R_{1,p}\times\cdots \times R_{1,p}}_{\text{$t_1$\ times}}\times\cdots\times\underbrace{R_{\ell,p}\times\cdots \times R_{\ell,p}}_{\text{$t_\ell$\ times}}\hookrightarrow\\
&\underbrace{\sym{p}\times\cdots \times \sym{p}}_{\text{$t_1$\ times}}\times\cdots\times\underbrace{\sym{p^\ell}\times\cdots \times \sym{p^\ell}}_{\text{$t_\ell$\ times}}\hookrightarrow \sym{n},
\end{align*}
where we have $n=\sum_{i=0}^\ell p^it_i$ for some non-negative integer $\ell$ such that all $t_0,\cdots,t_\ell$ are non-negative integers and $0\leq t_0<p$.
\end{thm}

Notice that the maximal elementary abelian $p$-subgroup occurred in Theorem \ref{T;Maximal subgroup} has $p$-rank $\sum_{i=1}^{\ell}it_{i}$. In particular, any elementary abelian $p$-subgroup of $\sym{n}$ has $p$-rank no more than  $\lfloor\frac{n}{p}\rfloor$.
\begin{lem}\label{L;Maximalsubgroup1}
Let $p=2$ and $2\mid n$. Let $E$ be a maximal elementary abelian $2$-subgroup of $\sym{n}$ with $2$-rank $\frac{n}{2}$. Then $E$ is $\sym{n}$-conjugate to some $K_\ell\times F_\ell$ where
$0\leq\ell\leq\lfloor\frac{n}{4}\rfloor$.
\end{lem}
\begin{proof}
By Theorem \ref{T;Maximal subgroup}, $E$ is $\sym{n}$-conjugate to an elementary abelian $2$-subgroup $(R_{1,2})^{t_1}\times\cdots\times (R_{m,2})^{t_m}\hookrightarrow
(\sym{2})^{t_1}\times\cdots\times(\sym{2^m})^{t_m}\hookrightarrow\sym{n}$ for some positive integer $m$ and some non-negative integers $t_1,\cdots,t_m$. We have $\sum_{i=1}^m2^it_i=n$ and $\sum_{i=1}^m it_i=\frac{n}{2}$ by the hypotheses. The two equations imply that $\sum_{i=1}^m 2^{i-1}t_i=\frac{n}{2}=\sum_{i=1}^m it_i$. For any integer $k$, note that $k<2^{k-1}$ if $2<k$. We therefore deduce that $t_i=0$ for all $2<i\leq m$. The lemma thus follows.
\end{proof}

\begin{lem}\label{L;sjtp=2}
Let $p=2$ and $2\mid n$. Then
\begin{enumerate}
\item[{\em (i)}] $D_1{\downarrow_{K_{n/4}}}$ has stable generic Jordan type $[1]^2$ when $n\equiv 0 \pmod4$.
\item[{\em (ii)}] $D_1{\downarrow_{E_{n/2}}}$ is generically free.
\end{enumerate}
\end{lem}
\begin{proof}
Let $K$ and $E$ be the elementary abelian $2$-subgroups $K_{n/4}$ and $E_{n/2}$ of $\sym{n}$ respectively. For (i), let $m:=\frac{n}{4}$ and $\cup_{i=1}^{2m}\{\alpha_i\}$ be a set of indeterminates. Let $\K:=\F(\alpha_1,\cdots,\alpha_{n/2})$ and $u_\alpha:=
1+\sum_{i=1}^m(\alpha_{2i-1}(k_{i,1}+1)+\alpha_{2i}(k_{i,2}+1))$. Let $N$ be the block matrix representing $u_\alpha+1$ on $D_1$ with respect to $\mathcal{B}_2^n$. When $12\leq n$, we have
$$N=(B_1,\ B(1,\alpha_3, \alpha_4),\ B(5,\alpha_5,\alpha_6),\cdots,B(n-11,\alpha_{(n/2)-3},\alpha_{(n/2)-2}),\ B_2).$$
Here $B(a,\beta,\gamma)$ denotes the $(n-2)\times4$ $\K$-matrix
\begin{equation*}
\begin{pmatrix}
\alpha_1 & \alpha_1 & \alpha_1 & \alpha_1\\
\alpha_2 & \alpha_2 & \alpha_2 & \alpha_2\\
0 & 0 & 0 & 0\\
\vdots & \vdots & \vdots & \vdots\\
0 & 0 & 0 & 0\\
\beta+\gamma & \beta & \gamma & 0\\
\beta & \beta+\gamma & 0 & \gamma\\
\gamma & 0 & \beta+\gamma& \beta\\
0 & \gamma & \beta & \beta+\gamma\\
\vdots & \vdots & \vdots & \vdots\\
0 & 0 & 0 & 0 \end{pmatrix},
\end{equation*}
where the integers $a$, $b$ have relations $a=4b+1$, $0\leq b\leq \frac{n-12}{4}$ and the symbols $\beta$, $\gamma$ occur from row $a+3$ to row $a+6$. The left two $(n-2)\times 3$ $\K$-matrices have definitions

\begin{equation*}
B_1:=\begin{pmatrix}
\alpha_2 & \alpha_1 & \alpha_1 +\alpha_2\\
\alpha_2 & \alpha_1 & \alpha_1 +\alpha_2\\
\alpha_2 & \alpha_1 & \alpha_1 +\alpha_2\\
0 & 0 & 0 \\
\vdots & \vdots & \vdots \\
0 & 0 & 0 \end{pmatrix},\end{equation*}
\begin{equation*}
B_2:=\begin{pmatrix}
\alpha_1 & \alpha_1+\alpha_{n/2} & \alpha_1+\alpha_{(n/2)-1}\\
\alpha_2 & \alpha_2+\alpha_{n/2} & \alpha_2+\alpha_{(n/2)-1}\\
0 & \alpha_{n/2} & \alpha_{(n/2)-1}\\
\vdots & \vdots & \vdots\\
0 & \alpha_{n/2} & \alpha_{(n/2)-1}\\
\alpha_{(n/2)-1}+\alpha_{n/2} & \alpha_{(n/2)-1}+\alpha_{n/2} & \alpha_{(n/2)-1}+\alpha_{n/2}\\
\alpha_{(n/2)-1} & \alpha_{(n/2)-1} & \alpha_{(n/2)-1}\\
\alpha_{n/2} & \alpha_{n/2} & \alpha_{n/2}
\end{pmatrix}.
\end{equation*}
When $n=4\ \text{or}\ 8$, the block matrix $N$ degenerates to be zero matrix or $(B_1,B_2)$ respectively. As stable generic Jordan type of $D_1{\downarrow_K}$ is independent of the choice of generators of $K$, we get (i) by counting the rank of $N$.

For (ii), suppose $D_1{\downarrow_{E}}$ is not generically free, we thus have $\mathrm{V}_{E}^{\#}(D_1{\downarrow_E})=\mathrm{V}_{E}^{\#}(\F)$ by Proposition \ref{P;sjt} (vi). Note that $D_1{\downarrow_{\sym{n-1}}}\cong D^{(n-2,1)}$ by \cite[Theorem 3.1]{Sheth}. Take $\tilde{E}$ to be $E_{(n/2)-1}$ and notice that $\tilde{E}$ is an elementary abelian $2$-subgroup of $\sym{n-1}\cap E$. We have $\mathrm{V}_{\tilde{E}}^{\#}(D^{(n-2,1)}{\downarrow_{\tilde{E}}})=\mathrm{V}_{\tilde{E}}^{\#}(D_1{\downarrow_{\tilde{E}}})=
\mathrm{V}_{\tilde{E}}^{\#}(\F)$ by the definition of rank variety of a module. It implies the following calculation:
$$ \frac{n}{2}-1=\dim\mathrm{V}_{\tilde{E}}^{\#}(\F)=\dim\mathrm{V}_{\tilde{E}}^{\#}(D^{(n-2,1)}{\downarrow_{\tilde{E}}})\leq c_{\sym{n-1}}(D^{(n-2,1)})\leq
\frac{n}{2}-2.$$
The first inequality of the calculation comes from the definition of complexity of a module. The second one is clear since the partition $(n-2,1)$ has $2$-weight $\frac{n}{2}-2$. The calculation shows an obvious contradiction. So $D_1{\downarrow_{E}}$ is generically free.
\end{proof}

\begin{lem}\label{L;nequiv4mod2}
$c_{\sym{n}}(D_1)=\frac{n}{2}-1$ when $p=2$ and $n\equiv 2 \pmod 4$.
\end{lem}
\begin{proof}
Let $\alpha:=(\alpha_1,\cdots,\alpha_{n/2})$ be a non-zero point of $\F^{\frac{n}{2}}$. Let $\ell$, $m$ be two non-negative integers such that $m=\frac{n-4\ell}{2}$. When $0\leq \ell\leq \frac{n-2}{4}$, let $u_{\alpha,\ell}$ be the element $1+\sum_{i=1}^\ell(\alpha_{2i-1}(k_{i,1}+1)+\alpha_{2i}(k_{i,2}+1))+\sum_{i=1}^m\alpha_{2\ell+i}(s_{\ell,i}+1)$. Let $N(\ell)$ be the corresponding matrix representing $u_{\alpha,\ell}+1$ on $D_1$ with respect to $\mathcal{B}_2^n$. We assign $E$ to be $K_\ell\times F_\ell$ and determine the corresponding $\mathrm{V}_{E}^{\#}(D_1{\downarrow_E})$ with respect to the $u_{\alpha,\ell}$ in the following cases.
\begin{enumerate}[\text{Case} 1:]
\item $\ell=0$
\end{enumerate}
In the case, when $10\leq n$, we have
$$ N(\ell)=(C_1,\ C(1,\alpha_3,\alpha_4),\ C(5,\alpha_5,\alpha_6),\cdots,C(n-9,\alpha_{(n/2)-2},\alpha_{(n/2)-1}),\ C_2).
$$
Here $C(a,\beta,\gamma)$ denotes the $(n-2)\times4$ $\F$-matrix
\begin{equation*}
\begin{pmatrix}
\alpha_1 & \alpha_1 & \alpha_1 & \alpha_1\\
0 & 0& 0& 0\\
\vdots & \vdots & \vdots & \vdots\\
0 & 0 & 0 & 0\\
\beta & \beta & 0 & 0\\
\beta & \beta & 0 & 0\\
0& 0 & \gamma& \gamma\\
0& 0 & \gamma& \gamma\\
\vdots & \vdots & \vdots & \vdots\\
0 & 0 & 0 & 0 \end{pmatrix},
\end{equation*}
where the integers $a$, $b$ have relations $a=4b+1$, $0\leq b\leq \frac{n-10}{4}$ and the symbols $\beta$, $\gamma$ occur from row $a+3$ to row $a+6$. The matrices $C_1$ and $C_2$ are the $(n-2)\times 3$ and $(n-2)\times 1$ $\F$-matrices where
\begin{equation*}
C_1:=\begin{pmatrix}
0 & \alpha_1 & \alpha_1 \\
0 & \alpha_2 & \alpha_2\\
0 & \alpha_2 & \alpha_2\\
\vdots & \vdots & \vdots \\
0 & 0 & 0 \end{pmatrix},
C_2:=\begin{pmatrix}
\alpha_1+\alpha_{n/2}\\
\alpha_{n/2}\\
\vdots\\
\alpha_{n/2}\\
0
\end{pmatrix}.
\end{equation*}
When $n=6$, the block matrix $N(\ell)$ degenerates to be $(C_1,C_2)$. By calculation of the rank of $N(\ell)$, one gets
$$\mathrm{V}_{E}^{\#}(D_1{\downarrow_E})=\{(\gamma_1,\cdots,\gamma_{n/2})\in \F^{\frac{n}{2}}:\ \sum_{i=1}^{\frac{n}{2}}\bar{\gamma_i}=0,\ \bar{\gamma_i}:=\prod_{\substack{j=1\\
                  j\neq i\\
                  } }^{\frac{n}{2}}\gamma_j\}.$$
Therefore, we have $\dim\mathrm{V}_{E}^{\#}(D_1{\downarrow_E})=\frac{n}{2}-1$.
\begin{enumerate}[\text{Case} 2:]
\item $\ell=\frac{n-2}{4}$
\end{enumerate}
In the case, when $10\leq n$, we get
$$N(\ell)=(B_1,\ B(1,\alpha_3,\alpha_4),\ B(5,\alpha_5,\alpha_6),\cdots,B(n-9,\alpha_{(n/2)-2},\alpha_{(n/2)-1}),\ \bar{B_2}).$$
The $B_1$, $B(a,\beta,\gamma)$, defined in the proof of Lemma \ref{L;sjtp=2} (i), are viewed as $\F$-matrices by regarding their entries $\alpha_1$, $\alpha_2$ as corresponding components of $\alpha$. Moreover, the integers $a$ and $b$ have relations $a=4b+1$, $0\leq b\leq \frac{n-10}{4}$. The matrix $\bar{B_2}$ is the $(n-2)\times 1$ $\F$-matrix
\begin{equation*}
\begin{pmatrix}
\alpha_1+\alpha_{n/2}\\
\alpha_2+\alpha_{n/2}\\
\alpha_{n/2}\\
\vdots\\
\alpha_{n/2}\\
0
\end{pmatrix}.
\end{equation*}
When $n=6$, the block matrix $N(\ell)$ degenerates to be $(B_1,\bar{B_2})$. By calculation of the rank of $N(\ell)$, one obtains
$$\mathrm{V}_{E}^{\#}(D_1{\downarrow_E})=\bigcup_{\substack{s=1\\
                  2\nmid s\\
                  }}^{\frac{n}{2}-2}\{(\gamma_1,\cdots,\gamma_{n/2})\in\F^{\frac{n}{2}}:\ \gamma_s=\gamma_{s+1}=0\}.$$
Therefore, we get $\dim \mathrm{V}_{E}^{\#}(D_1{\downarrow_E})=\frac{n}{2}-2$.
\begin{enumerate}[\text{Case} 3:]
\item $0<\ell<\frac{n-2}{4}$
\end{enumerate}
In the case, when $1<\ell$, we obtain $N(\ell)=(B_1,\ B,\ D,\ \bar{B_2})$, where
\begin{align*}
& B:=(B(1,\alpha_3,\alpha_4),B(5,\alpha_5,\alpha_6),\cdots,B(4\ell-7,\alpha_{2\ell-1},\alpha_{2\ell})),\\ &D:=(D(4\ell-3,\alpha_{2\ell+1},\alpha_{2\ell+2}),\cdots,D(n-9,\alpha_{(n/2)-2},\alpha_{(n/2)-1})).
\end{align*}
Here $D(a,\beta,\gamma)$ denotes the $(n-2)\times4$ $\F$-matrix
\begin{equation*}
\begin{pmatrix}
\alpha_1 & \alpha_1 & \alpha_1 & \alpha_1\\
\alpha_2 & \alpha_2 & \alpha_2 & \alpha_2\\
0 & 0& 0& 0\\
\vdots & \vdots & \vdots & \vdots\\
0 & 0 & 0 & 0\\
\beta & \beta & 0 & 0\\
\beta & \beta & 0 & 0\\
0& 0 & \gamma& \gamma\\
0& 0 & \gamma& \gamma\\
\vdots & \vdots & \vdots & \vdots\\
0 & 0 & 0 & 0 \end{pmatrix},
\end{equation*}
where the integers $a$, $b$ have relations $a=4b+1$, $\ell-1\leq b\leq \frac{n-10}{4}$ and the symbols $\beta$, $\gamma$ occur from row $a+3$ to row $a+6$. The definitions of other components of $N(\ell)$ are from Case $2$. Given a component $B(a',\beta,\gamma)$ of $B$, the integers $a'$, $b'$ have relations $a'=4b'+1$ and $0\leq b'\leq \ell-2$. When $\ell=1$, the block matrix $N(\ell)$ degenerates to be $(B_1,D,\bar{B_2})$. Define two sets

\begin{align*}
& \mathrm{V}_1:=\bigcup_{\substack{s=1\\
                  2\nmid s\\
                  }}^{\frac{n}{2}-2}\{(\gamma_1,\cdots,\gamma_{n/2})\in\F^{\frac{n}{2}}:\ \gamma_s=\gamma_{s+1}=0\},\\
& \mathrm{V}_2(\ell):=\{(\gamma_1,\cdots,\gamma_{n/2})\in\F^{\frac{n}{2}}:\ \sum_{i=2\ell+1}^{\frac{n}{2}}\bar{\gamma_i}=0,\ \bar{\gamma_i}:={\prod_{\substack{j=2\ell+1\\
                  j\neq i\\
                  }}^{\frac{n}{2}}\gamma_j}\}.
\end{align*}
According to calculation of the rank of $N(\ell)$, we know that $\mathrm{V}_{E}^{\#}(D_1{\downarrow_E})=V_1\cup V_2(\ell)$. Therefore, we have $\dim\mathrm{V}_{E}^{\#}(D_1{\downarrow_E})=\frac{n}{2}-1$. The lemma now follows by combining all cases discussed above, Lemma \ref{L;Maximalsubgroup1} and the definition of complexity of a module.
\end{proof}
\begin{prop}\label{P;(n-1,1)2}
Let $(n-1,1)\in\mathfrak{Reg}_p(n)$ having $p$-weight $w$. We have
\[c_{\sym{n}}(D_1)=\begin{cases}
w-1, & \text{if}\ p=2,\ n\equiv 2\pmod 4,\\
w,  &\text{otherwise}.\\
\end{cases} \]
\end{prop}
\begin{proof}
The case $p\nmid n$ is trivial since we have $S_1\cong D_1$ by Theorem \ref{T;James}. The result follows by Lemma \ref{L;YongModules} (iii). When $p\neq2$ and $p\mid n$, we have $\dim_\F D_1\equiv p-2 \pmod p$ which implies that $c_{\sym{n}}(D_1)=w$. For the case where $p=2$ and $2\mid n$, note that $w=\frac{n}{2}$. We also get desired result by Lemma \ref{L;sjtp=2} (i), Proposition \ref{P;sjt} (vii) and Lemma \ref{L;nequiv4mod2}. The proof of the proposition is now complete.
\end{proof}

\begin{prop}\label{P;CD(n-2,2)}
Let $(n-2,2)\in\mathfrak{Reg}_p(n)$ having $p$-weight $w$. If $n\neq 6$, we have
\begin{align*}
c_{\sym{n}}(D_2)=\begin{cases}
w-1, &\text{if}\ p=2,\ n=5,\\
w, &\text{otherwise}.\\
\end{cases}
\end{align*}
\end{prop}
\begin{proof}
When $2<p$, by Theorem \ref{T;James}, we have
\begin{equation}
S_2\sim\begin{cases}
D_2, &\text{if}\ n\not\equiv 1,\ 2\ \ \pmod p,\\
D_2+D^{(n)}, &\text{if}\ n\equiv 1\ \ \ \ \ \ \pmod p,\\
D_2+D_1, &\text{if}\ n\equiv 2\ \ \ \ \ \ \pmod p.
\end{cases}
\end{equation}
The situation $n\not\equiv1,\ 2 \pmod p$ is trivial by Lemma \ref{L;YongModules} (iii). By $(3.2)$ and $(2.4)$, one gets $\dim_\F D_2\equiv p-2\pmod p$ when $n\equiv 1\ \text{or}\ 2\pmod p$, which forces $c_{\sym{n}}(D_2)=w$.

We deal with the case $p=2$. When $n\equiv 1 \pmod 4$ and $5<n$, take $E$ to be $E_{(n-1)/2}$ and note that $w=\frac{n-1}{2}$. By Theorem \ref{T;James} and Lemma \ref{L;sjt(n-2,2)2}, we have $S_2\sim D_2+D^{(n)}$ and $S_2{\downarrow_E}$ has stable generic Jordan type $[1]^\frac{n-3}{2}$. As $1<\frac{n-3}{2}$, the two facts force that $D_2{\downarrow_E}$ is not generically free by Proposition \ref{P;sjt} (iv) (c). We thus get that $c_{\sym{n}}(D_2)=w$ by Proposition \ref{P;sjt} (vii). For the left case $p=2$ and $n=5$, we have $c_{\sym{5}}(D^{(3,2)})=1$ by \cite[Example 1]{Uno}. When $n\equiv 3 \pmod 4$, we have $S_2\cong D_2$ by Theorem \ref{T;James}. It is trivial by Lemma \ref{L;YongModules} (iii). When $2\mid n$, assign $E$ to be $E_{n/2}$ and note that $w=\frac{n}{2}$. When $n\equiv 0\pmod 4$, by Theorem \ref{T;James}, we have $S_2\sim D_2+D_1$. Notice that $D_1{\downarrow_E}$ is generically free by Lemma \ref{L;sjtp=2} (ii). It implies that both $D_2{\downarrow_E}$ and $S_2{\downarrow_E}$ have stable generic Jordan type $[1]^{\frac{n}{2}-2}$ by Proposition \ref{P;sjt} (iv) (a) and Lemma \ref{L;p=2,sjt}. We conclude that $D_2{\downarrow_E}$ is not generically free since $n\neq 4$. Therefore, we have $c_{\sym{n}}(D_2)=w$ by Proposition \ref{P;sjt} (vii). Finally, when $n\equiv 2 \pmod 4$, we have $S_2\sim D_2+D_1+D^{(n)}$ by Theorem \ref{T;James}. Define $R$ to be $\text{Rad}(S_2)$. By \cite[Theorem 6 (ii)]{Peel}, we have the short exact sequences
\begin{align*}
0\rightarrow R{\downarrow_E}\rightarrow S_2{\downarrow_E}\rightarrow D_2{\downarrow_E}\rightarrow 0,\ 0\rightarrow D_1{\downarrow_E}\rightarrow R{\downarrow_E}\rightarrow D^{(n)}{\downarrow_E}\rightarrow 0.
\end{align*}
Suppose $D_2{\downarrow_E}$ is generically free. Note that both $R{\downarrow_E}$ and $S_2{\downarrow_E}$ have stable generic Jordan type $[1]^{\frac{n}{2}-2}$ by the short exact sequence on the left-hand side, Proposition \ref{P;sjt} (iv) (c) and Lemma \ref{L;p=2,sjt}. Using the other short exact sequence, Lemma \ref{L;sjtp=2} (ii) and Proposition \ref{P;sjt} (iv) (a), we know $R{\downarrow_E}$ and $D^{(n)}{\downarrow_E}$ have same stable generic Jordan type. The fact implies that $\frac{n}{2}-2=1$. It is a contradiction since $n\neq 6$. Hence $D_2{\downarrow_E}$ is not generically free, which implies that $c_{\sym{n}}(D_2)=w$ by Proposition \ref{P;sjt} (vii). The proof is now finished.
\end{proof}
\begin{rem}\label{R;p=2,n=6}
When $p=2$ and $n=6$, let $E$ and $F$ be $F_0$ and $K_1\times F_1$ respectively. By Theorem \ref{T;Maximal subgroup}, they are exactly all maximal elementary abelian $2$-subgroups of $\sym{6}$ up to $\sym{6}$-conjugation. By brute force computation, $\dim \mathrm V_{E}^\#(D^{(4,2)}{\downarrow_{E}})=1$ and $\dim \mathrm V_{F}^\#(D^{(4,2)}{\downarrow_{F}})=2$. The computation shows that $c_{\sym{6}}(D^{(4,2)})=2$ by the definition of complexity of a module.
\end{rem}

Theorem \ref{T;A} is now proved by combining Propositions \ref{P;(n-1,1)2}, \ref{P;CD(n-2,2)} and Remark \ref{R;p=2,n=6}.

\section{Proof of Theorem \ref{T;C}}
The section focuses on proving Theorem \ref{T;C}. We first show some needed results and then finish the proof of Theorem \ref{T;C}. For convenience, in the following discussion, we fix $r$ to be the reminder of $n$ when $n$ is divided by $p$. Moreover, throughout the whole section, let $s$ be an integer such that $0<s<p$.

We begin with a theorem which is useful in calculation of dimensions of simple modules of symmetric groups labelled by two-part partitions.
\begin{thm}\label{T; Lucas}\cite[Lucas Theorem]{Lucas}\label{Lucas}
Let two non-negative integers $a$ and $b$ have the $p$-adic sums $\sum_{0\leq i}p^{i}a_{i}$ and $\sum_{ 0\leq i}p^{i}b_{i}$ respectively. Then ${a\choose b}\equiv \prod_{0\leq i}{a_i\choose b_i} \pmod p$.
\end{thm}
The following lemma is a property of $p$-restricted two-part partitions.

\begin{lem}\label{L:pweight} \cite[Lemma 5.2]{JLW}\label{JLW}
Given a $p$-restricted partition $(a, b)$, it has $p$-weight either zero or one. Furthermore, it has $p$-weight one if and only if $a-b<p-1\leq a$.
\end{lem}

For a partition $(n-m,m)$, note that it has $p$-core $(r)$ if $p\mid m$. The following lemma determines its $p$-core if $p\nmid m$.
\begin{lem}\label{L:pcore}
Let $m$ be a non-negative integer such that $2mp+2s\leq n$ and let $\lambda:=(n-pm-s,pm+s)$.
\begin{enumerate}
\item [\em(i)]When $p=2$, $\kappa_2(\lambda)=\begin{cases} \varnothing, &\text{if}\ r=0,\\ (2,1), & \text{if}\ r=1.\end{cases}$
\item [\em(ii)] When $2<p$ and $0<s\leq \frac{p-1}{2}$, \[\kappa_{p}(\lambda)=\begin{cases}
(p+r-s,s), & \text{if}\ 0\leq r<s-1,\\
(s-1,r-s+1), & \text{if}\ s-1\leq r<2s-1,\\
(p+s-1,s), & \text{if}\ r=2s-1,\\
(r-s,s), & \text{if}\ 2s-1<r\leq p-1.
\end{cases} \]

\item [\em(iii)] When $2<p$ and $s=\frac{p+1}{2}$, \[ \kappa_{p}(\lambda)=\begin{cases}
(\frac{3p-1}{2},\frac{p+1}{2}), & \text{if}\ r=0,\\
(\frac{p-1}{2}+r,\frac{p+1}{2}), & \text{if}\ 0<r<\frac{p-1}{2}, \\
(\frac{p-1}{2},r-\frac{p-1}{2}), & \text{if}\ \frac{p-1}{2}\leq r\leq p-1.\\

\end{cases} \]
\item [\em(iv)] When $3<p$ and $\frac{p+1}{2}<s<p$, \[ \kappa_{p}(\lambda)=\begin{cases}
(s-1,p+r-s+1), &\text{if}\ 0\leq r<2s-p-1,\\
(p+s-1,s), & \text{if}\ r=2s-p-1,\\
(p+r-s,s), & \text{if}\ 2s-p-1<r<s-1,\\
(s-1,r-s+1), &\text{if}\ s-1\leq r\leq p-1.\\
\end{cases} \]
\end{enumerate}
\end{lem}
\begin{proof}
One removes all horizontal $p$-hooks of $\lambda$ successively to make it become a $p$-restricted partition and then uses Lemma \ref{L:pweight} to compute directly.
\end{proof}

A direct corollary of Lemma \ref{L:pcore} is given as follows.

\begin{cor}\label{L;Findpcore}
A partition $(n-m,m)$ has $p$-core $(r)$ only if $m\equiv 0\ \text{or}\ r+1\pmod p$.
\end{cor}
For the reminder of the section, we fix $p$ to be an odd prime. The following lemma is an application of Lemma \ref{L:pcore} and Theorem \ref{T;James}.
\begin{lem}\label{L;Dstructure}
Let $2s\leq n$ and $\lambda:=(n-s,s)$.
\begin{enumerate}
\item [\em(i)] When $0<s\leq \frac{p+1}{2}$,
\[S^{\lambda}\sim\begin{cases} D^{\lambda}+D^{\mu}, &\text{if}\ s-1\leq r<2s-1,\\
D^{\lambda}, & \text{otherwise},
\end{cases}\]
where $\mu:=(n-r+s-1,r-s+1)$.
\item [\em(ii)] When $3<p$ and $\frac{p+1}{2}<s<p$,
\[S^{\lambda}\sim\begin{cases} D^{\lambda}+D^{\nu}, &\text{if}\ 0\leq r<2s-p-1,\\
D^{\lambda}, &\text{if}\ 2s-p-1\leq r<s-1,\\
D^{\lambda}+D^{\mu}, &\text{if}\ s-1\leq r\leq p-1,
\end{cases}\]
where $\mu:=(n-r+s-1,r-s+1)$ and $\nu:=(n-p-r+s-1,p+r-s+1)$.
\end{enumerate}
\end{lem}

\begin{lem}\label{L; YP1}
Let $i$ be an integer such that $1\leq i\leq p-1$. Let $n\equiv ip \pmod {p^{2}}$ and $\lambda$, $\mu$, $\nu$ be the partitions $(n-p,p)$, $(n-p+1,p-1)$, $(n-1,1)$ respectively. Take $E$ to be $E_{n/p}$. Then
\begin{enumerate}
\item [\em(i)] $Y^{\lambda}{\downarrow_E}$ has stable generic Jordan type $[1]^{\frac{n}{p}-1}$.
\item [\em(ii)]There exists a short exact sequence $0\rightarrow S^{\lambda}\rightarrow Y^\lambda\rightarrow S^{\nu}\rightarrow 0$.
\end{enumerate}
\end{lem}
\begin{proof}
We claim that $M^{\lambda}\cong Y^{\lambda}\oplus M\oplus Y^{(n)}$ where $M\mid M^{\mu}$. Let $m$ be an integer such that $2m\leq n$ and $0<m<p$ and $\eta:=(n-m,m)$. We have $p-m-1\subseteq_p n-2m$ if $p-m\subseteq_p n-2m$. By Theorem \ref{T;Henke}, the relation implies that $Y^{\eta}\mid M^{\mu}$ if $Y^{\eta}\mid M^{\lambda}$. Using Theorem \ref{T;Henke} again, we also have  $Y^{(n)}\mid M^{\lambda}$ as $p\subseteq_p n$. The claim is shown. Note that $M^{\lambda}{\downarrow_E}$ has stable generic Jordan type $[1]^{\frac{n}{p}}$ and $M^{\mu}{\downarrow_E}$ is generically free by Lemma \ref{PModules}. Then (i) follows by the claim and Proposition \ref{P;sjt} (iii).

For (ii), by Corollary \ref{L;Findpcore}, observe that $\kappa_p(\lambda)=\kappa_p((n))$ and $\nu$ is the unique partition of $n$ such that $\lambda\lhd\nu\lhd(n)$ and $\kappa_p(\nu)=\kappa_p(\lambda)$. As $p-1\nsubseteq_p n-2$, we get $Y^\nu\nmid M^{\lambda}$ by Theorem \ref{T;Henke}. These facts and the claim
imply the existence of the desired short exact sequence by Lemma \ref{R;Spechtfiltration} (i).
\end{proof}

\begin{lem}\label{L; YP2}
Let $n\equiv p+r\pmod {p^{2}}$ and $\lambda$, $\mu$, $\nu$ be the partitions $(n-p-s,p+s)$, $(n-p-s+1,p+s-1)$ and $(n-s,s)$ respectively where $r=2s-1$. Then 
\begin{enumerate}
\item [\em(i)] $M^{\lambda}\cong Y^{\lambda}\oplus M$ where $M\mid M^{\mu}$. Moreover, we have $S^{\nu}\cong D^{\nu}\cong Y^{\nu}$.
\item [\em(ii)] there exists a short exact sequence $0\rightarrow S^{\lambda}\rightarrow Y^\lambda\rightarrow S^{\nu}\rightarrow 0$.
\end{enumerate}
\end{lem}
\begin{proof}
Let $m$ be an integer such that $2m\leq n$ and $0\leq m<p+s$ and $\eta:=(n-m,m)$. We claim that $p+s-m-1\subseteq_p n-2m$ if $p+s-m\subseteq_p n-2m$. Indeed, we have above relation if $0\leq m<s$. Note that $p\not\subseteq_p n-2s$ since $r=2s-1$. If $s<m<p+s$, we have $0<p+s-m<p$ and therefore $p+s-m-1\subseteq_p n-2m$ if $p+s-m\subseteq_p n-2m$. The claim is proved. By the claim and Theorem \ref{T;Henke}, we get $Y^{\eta}\mid M^{\mu}$ if $Y^{\eta}\mid M^{\lambda}$. So the first isomorphic formula follows. As $s=\frac{r+1}{2}<\frac{p+1}{2}$, we also have $S^{\nu}\cong D^{\nu}\cong Y^{\nu}$ by Lemmas \ref{L;Dstructure}, \ref{L;YongModules} (i). The proof of (i) is complete.

For (ii), by Lemma \ref{L:pcore}, observe that $\kappa_p(\lambda)\neq\kappa_p((n))$ and $\nu$ is the unique partition of $n$ such that $\lambda\lhd\nu\lhd (n)$ and $\kappa_p(\nu)=\kappa_p(\lambda)$. Using Theorem \ref{T;Henke}, we know $Y^{\nu}\nmid M^{\lambda}$ as $p\not\subseteq_p n-2s$. These facts imply the existence of the desired short exact sequence by Lemma \ref{R;Spechtfiltration} (ii).
\end{proof}

\begin{lem}\label{L;p=3,YP2}
Let $n\equiv p+r\pmod {p^{2}}$ and $\lambda$, $\mu$ be the partitions $(n-p-s,p+s)$, $(n-s,s)$ respectively where $r=2s-1$. Take $E$ to be $E_{(n-p-r)/p}$. Let $Y^{\lambda}{\downarrow_E}$ and $Y^{\mu}{\downarrow_E}$ have stable generic Jordan types $[1]^a$ and $[1]^b$ respectively. Then $2b<a$.
\end{lem}
\begin{proof}
Denote the partition $(n-p-s+1,p+s-1)$ by $\nu$. When $p=3$, we have $r=s=1$. In the case, note that $S^\mu\cong D^\mu$ and $b=3$ by Lemmas \ref{L; YP2} (i) and \ref{L;(n-1,1)sjt}. From Theorem \ref{T;Henke}, we get $M^\lambda\cong Y^\lambda\oplus Y^\nu\oplus Y^{(n)}$ and $M^\nu\cong Y^\nu\oplus Y^\mu\oplus Y^{(n)}$. Using the two isomorphic formulae, Lemma \ref{PModules} and Proposition \ref{P;sjt} (iii), we obtain $a=n-4$. The inequality thus holds as $n\equiv 4\pmod 9$.

When $3<p$, let $M^{\lambda}{\downarrow_E}$, $M^{\mu}{\downarrow_E}$, $M^{\nu}{\downarrow_E}$ have stable generic Jordan types $[1]^{m_\lambda}$, $[1]^{m_\mu}$, $[1]^{m_\nu}$ respectively. Using Lemma \ref{PModules}, we can obtain $m_{\lambda}= {p+r\choose p+s}+\frac{n-p-r}{p}{p+r\choose s}$, $m_{\nu}={p+r\choose p+s-1}+\frac{n-p-r}{p}{p+r\choose s-1}$ and $m_\mu={p+r\choose s}$. Notice that $m_{\lambda}-m_{\nu}\leq a$ and $b\leq m_\mu$ by the first isomorphic formula of Lemma \ref{L; YP2} (i) and Proposition \ref{P;sjt} (iii). To prove that $2b<a$, it is enough to show that $2m_\mu<m_\lambda-m_\nu$. We need to transfer the inequality to an easier one and justify the new one. The transformation is given as follows.
\begin{align*}
&2m_\mu<m_\lambda-m_\nu \\
&\Leftrightarrow 2{p+2s-1\choose s}<(\frac{n-p-2s+1}{p}-1)({p+2s-1\choose s}-{p+2s-1\choose s-1})\\
&\Leftarrow 2{p+2s-1\choose s}<(p-1)({p+2s-1\choose s}-{p+2s-1\choose s-1})\\
&\Leftrightarrow 2s<(p-3)p
\end{align*}
In the transformation, the second row follows by using $r=2s-1$. Note that ${p+2s-1\choose s-1}<{p+2s-1\choose s}$ and the third row follows by taking $n$ to be $p^2+p+2s-1$. Finally, we have the last row by direct calculation. The correctness of the inequality in the last row is clear since we have $3<p$ and $0<s<p$. The lemma follows.
\end{proof}

One can mimic the proof of Lemma \ref{L; YP2} to show the following lemma.
\begin{lem}\label{L;YP3}
Let $n\equiv2p+r\pmod {p^{2}}$ and $\lambda$, $\mu$, $\nu$ be the partitions $(n-p-s,p+s)$, $(n-p-s+1,p+s-1)$ and $(n-s,s)$ respectively where $r=2s-p-1$. Then
\begin{enumerate}
\item [\em(i)]$M^{\lambda}\cong Y^{\lambda}\oplus M$ where $M\mid M^{\mu}$. Moreover, we have $S^{\nu}\cong D^{\nu}\cong Y^{\nu}$.
\item [\em(ii)]There exists a short exact sequence $0\rightarrow S^{\lambda}\rightarrow Y^\lambda\rightarrow S^{\nu}\rightarrow 0$.
\end{enumerate}
\end{lem}

\begin{lem}\label{L;p=3,YP3}
Let $n\equiv 2p+r\pmod {p^{2}}$ and $\lambda$, $\mu$ be the partitions $(n-p-s,p+s)$, $(n-s,s)$ respectively where $r=2s-p-1$. Take $E$ to be $E_{(n-2p-r)/p}$. Let $Y^{\lambda}{\downarrow_E}$ and $Y^{\mu}{\downarrow_E}$ have stable generic Jordan types $[1]^a$ and $[1]^b$ respectively. Then $2b<a$.
\end{lem}
\begin{proof}
When $p=3$, we have $r=0$ and $s=2$. We get that $S^\mu\cong D^\mu$ and $b=9$ by Lemmas \ref{L;YP3} (i) and  \ref{L;p|n(n-2,2)}. A calculation, similar to the one of Lemma \ref{L;p=3,YP2}, shows that $a=3n-18$. The inequality holds as $n\equiv 6\pmod 9$. One can follow the proof of Lemma \ref{L;p=3,YP2} to write a similar proof for the case $3<p$. The lemma follows.
\end{proof}

We now begin to finish the proof of Theorem \ref{T;C}. The entire proof of Theorem \ref{T;C} will be divided into the following propositions.

\begin{prop}\label{P;Less p}
Let $2s\leq n$ and $\lambda:=(n-s,s)$ have $p$-weight $w$. Then $c_{\sym{n}}(D^{\lambda})=w$.
\end{prop}
\begin{proof}
The situation $S^{\lambda}\cong D^{\lambda}$ is trivial by Lemma \ref{L;YongModules} (iii). Ruling out the trivial situation by Lemma \ref{L;Dstructure}, we deal with the left situations in the following cases.
\begin{enumerate}[\text{Case} 1:]
\item $0<s\leq\frac{p+1}{2}$, $s-1\leq r<2s-1$ or $3<p$, $\frac{p+1}{2}<s<p$, $s-1\leq r\leq p-1$
\end{enumerate}
Note that $r<2s-1$ by the conditions of the case. In the case, denote the partition $(n-r+s-1,r-s+1)$ by $\mu$ and note that $S^{\lambda}\sim D^{\lambda}+D^{\mu}$, $S^{\mu}\cong D^{\mu}$ by Lemma \ref{L;Dstructure}. Using the relations, $(2.4)$ and Theorem \ref{T; Lucas}, we have $\dim_\F D^{\lambda}\equiv\frac{2(r-2s+1)r!}{(r-s+1)!s!}\pmod p$. Observe that $r-2s+1<0$ since $r<2s-1$. As $s-1\leq r$, we also have $2s-1=s+s-1<p+r$, which means that $-p<r-2s+1$. We thus conclude that $p\nmid \dim_\F D^{\lambda}$ and $c_{\sym{n}}(D^{\lambda})=w$.
\begin{enumerate}[\text{Case} 2:]
\item $3<p$, $\frac{p+1}{2}<s<p$, $0\leq r<2s-p-1$
\end{enumerate}
Take $E$ to be $E_{(n-p-r)/p}$ and note that $w=\frac{n-p-r}{p}$ by Lemma \ref{L:pcore}. In the case, denote the partition $(n-p-r+s-1,p+r-s+1)$ by $\nu$. By Lemma \ref{L:pcore}, observe that $\kappa_p(\lambda)\neq \kappa_p((n))$ and $\nu$ is the unique partition of $n$ such that $\lambda\lhd\nu\lhd (n)$ and $\kappa_p(\nu)=\kappa_p(\lambda)$. Using the inequality $0\leq r<2s-p-1$ and Theorem \ref{T;Henke}, we have $Y^{\nu}\nmid M^{\lambda}$. By Lemma \ref{R;Spechtfiltration} (ii), these facts imply that there exists a short exact sequence $0\rightarrow S^{\lambda}\rightarrow Y^{\lambda}\rightarrow S^{\nu}\rightarrow 0$. By Lemmas \ref{L;Dstructure} and \ref{L;YongModules} (i), we also have $S^{\lambda}\sim D^{\lambda}+D^{\nu}$ and $S^{\nu}\cong D^{\nu}\cong Y^{\nu}$.
Note that the vertices of $Y^{\lambda}$ are $\sym{n}$-conjugate to the Sylow $p$-subgroups of $\sym{n-2p-r}$ and the vertices of $Y^{\nu}$ are $\sym{n}$-conjugate to the Sylow $p$-subgroups of $\sym{n-p-r}$ by Theorem \ref{T; Youngvertices} and Lemma \ref{L:pcore}. Moreover, we know that $E$ is not contained in $\sym{n-2p-r}$ up to $\sym{n}$-conjugation. We thus deduce that $Y^{\lambda}{\downarrow_E}$ is generically free while $Y^{\nu}{\downarrow_E}$ is not generically free by Lemma \ref{L;YongModules} (ii). Let $Y^\nu{\downarrow_E}$ have stable generic Jordan type $[1]^a$ where $0<a$. Suppose that $D^{\lambda}{\downarrow_E}$ is generically free. We apply Proposition \ref{P;sjt} (iv) (b) to the short exact sequence $0\rightarrow S^{\lambda}{\downarrow_E}\rightarrow Y^{\lambda}{\downarrow_E}\rightarrow S^{\nu}{\downarrow_E}\rightarrow 0$ and get that $S^{\lambda}{\downarrow_E}$ has stable generic Jordan type $[p-1]^{a}$. However, we can also apply  Proposition \ref{P;sjt} (iv) (c) to the short exact sequence $0\rightarrow D^{\nu}{\downarrow_E}\rightarrow S^{\lambda}{\downarrow_E}\rightarrow D^{\lambda}{\downarrow_E}\rightarrow 0$ and get that $S^{\lambda}{\downarrow_E}$ has stable generic Jordan type $[1]^{a}$. It is a contradiction. We therefore conclude that $D^{\lambda}{\downarrow_E}$ is not generically free. The proposition follows by applying Proposition \ref{P;sjt} (vii) to $D^{\lambda}{\downarrow_E}$.
\end{proof}

When $2p\leq n$, we turn to determine $c_{\sym{n}}(D^{(n-p,p)})$. Note that a partition $(n-m,m)$ has $p$-weight $\frac{n-r}{p}$ if $p\mid m$.

\begin{lem}\label{L;Pblock1}
Let $2p\leq n$, $r=0$ and $\lambda:=(n-p,p)$. Then $c_{\sym{n}}(D^{\lambda})=\frac{n}{p}$.
\end{lem}
\begin{proof}
Denote the partition $(n-1,1)$ by $\mu$. By Corollary \ref{L;Findpcore} and Theorem \ref{T;James}, we have
\begin{align*}
\tag{4.1}S^{\lambda}\sim \begin{cases} D^{\lambda}+D^{\mu}+D^{(n)}, &\text{if}\ n\equiv p\pmod {p^{2}},\\
D^{\lambda}+D^{\mu}, & \text{otherwise}.\end{cases}
\end{align*}
Let $i$ be an integer such that $0\leq i\leq p-1$. By $(4.1)$, $(2.4)$, Theorem \ref{T; Lucas},
\begin{align*}
\dim_\F D^{\lambda}\equiv\begin{cases} 2 \ \ \ \ \ \pmod p, &\text{if}\ n\equiv p\ \pmod {p^{2}},\\
i+2 \pmod p, &\text{if}\ n\equiv ip \pmod {p^{2}},\ i\neq 1. \end{cases}
\end{align*}
So we get that $p\nmid \dim_\F D^{\lambda}$ and $c_{\sym{n}}(D^{\lambda})=\frac{n}{p}$ unless $3<p$ and $n\equiv (p-2)p \pmod {p^{2}}$.

When $3<p$ and $n\equiv (p-2)p \pmod {p^{2}}$, take $E$ to be $E_{n/p}$ and denote the partitions $(n-p-1,p)$, $(n-p-1,p+1)$, $(n-p-1,p,1)$ by $\nu$, $\eta$, $\zeta$ respectively. From Corollary \ref{L;Findpcore} and Theorem \ref{T;James}, we know $S^{\nu}\cong D^{\nu}$, $S^{\lambda}\sim D^{\lambda}+D^{\mu}$, $S^{\eta}\sim D^{\eta}+D^{\lambda}$. By Branching Rule and block decomposition, we have $D^{\nu}{\uparrow^{\sym{n}}}\cong S^{\nu}{\uparrow^{\sym{n}}}\cong M\oplus S^{\zeta}$ where $M\sim S^{\lambda}+S^{\eta}$. Using \cite[Theorem 11.2.8]{A.K}, one obtains $D^{\nu}{\uparrow^{\sym{n}}}\cong N\oplus D^{\zeta}$, where $N$ is a self-dual and indecomposable $\F\sym{n}$-module whose head and socle are isomorphic to $D^{\lambda}$. Note that $S^\zeta\cong D^\zeta$ by \cite[Theorem 7.3.23]{GJ3}. We deduce that $M\cong N$ from the isomorphic formulae of $D^\nu{\uparrow^{\sym{n}}}$ by Krull-Schmidt Theorem. Therefore, we can conclude that $N$ has three Loewy layers and its heart is isomorphic to $D^{\mu}\oplus D^{\eta}$.
Note that $D^\mu{\downarrow_E}$ is not generically free by counting dimension and Proposition \ref{P;sjt} (i). Denote $D^{\mu}\oplus D^{\eta}$ by $D$ and deduce from Proposition \ref{P;sjt} (iii) that $D{\downarrow_E}$ is not generically free. One also observes that $D^{\nu}{\uparrow^{\sym{n}}}{\downarrow_E}$ is generically free by Mackey Formula and Proposition \ref{P;sjt} (iii), (v). The fact implies that $N{\downarrow_E}$ is generically free by Proposition \ref{P;sjt} (iii). Now define $R$ to be $\text{Rad}(N)$ and suppose that $D^{\lambda}{\downarrow_E}$ is generically free. We can apply Proposition \ref{P;sjt} (iv) (c) to the short exact sequence $ 0\rightarrow R{\downarrow_E}\rightarrow N{\downarrow_E}\rightarrow D^{\lambda}{\downarrow_E}\rightarrow 0$ and get that $R{\downarrow_E}$ is generically free. We thus apply Proposition \ref{P;sjt} (iv) (a) to the short exact sequence $0\rightarrow D^{\lambda}{\downarrow_E}\rightarrow R{\downarrow_E}\rightarrow D{\downarrow_E}\rightarrow 0$ and deduce that $D{\downarrow_E}$ is generically free. However, it is a contradiction since we have known that $D{\downarrow_E}$ is not generically free. Therefore, $D^{\lambda}{\downarrow_E}$ is not generically free, which implies that $c_{\sym{n}}(D^\lambda)=\frac{n}{p}$ by Proposition \ref{P;sjt} (vii). The lemma follows.
\end{proof}
\begin{lem}\label{L;Sdecomposition}
Let $m$ be a non-negative integer and let $\lambda$ and $\mu$ be the partitions $(n-mp,mp)$, $(n-mp-r,mp)$ respectively. Then $D^{\mu}\mid D^\lambda{\downarrow_{\sym{n-r}}}$ if $0<r<p-1$.
\end{lem}
\begin{proof}
The case $m=0$ is trivial. The case where $p=3$ and $r=1$ is also clear from Theorem \ref{T;MBR}. For any integer $a$ satisfying condition $0\leq a\leq r$, denote the partition $(n-mp-a,mp)$ by $\lambda_a$. We deal with the left case where $3<p$ and $0<m$. When $0<r<p-2$, let $t$ be an integer such that $0\leq t<r$. We have $0<p-2-(r-t)<p-2$. The inequality implies that $p-2\not\equiv r-t-1\pmod p$ and $p-2\not\equiv r-t\pmod p$. The fact and Theorem \ref{T;MBR} say that $D^{\lambda_{t+1}}\mid D^{\lambda_t}{\downarrow_{\sym{n-t-1}}}$. Therefore, as $\lambda_0=\lambda$ and $\lambda_r=\mu$, we get $D^\mu\mid D^\lambda{\downarrow_{\sym{n-r}}}$ by induction. When $r=p-2$, by Theorem \ref{T;MBR}, we have $D^\lambda{\downarrow_{\sym{n-1}}}\cong D^{\lambda_1}$. Note that $n-1\equiv r-1=p-3 \pmod p$. We get $D^{\mu}\mid D^{\lambda_1}{\downarrow_{\sym{n-r}}}$ by the situation $0<r<p-2$. The lemma follows by noting that $D^\lambda{\downarrow_{\sym{n-r}}}\cong D^{\lambda_1}{\downarrow_{\sym{n-r}}}$.
\end{proof}

\begin{prop}\label{P;equal p}
Let $2p\leq n$ and $\lambda:=(n-p,p)$. Then $c_{\sym{n}}(D^{\lambda})=\frac{n-r}{p}$.
\end{prop}
\begin{proof}
Denote the partition $(n-p-r,p)$ by $\mu$. We are left with the case $r\neq 0$ by Lemma \ref{L;Pblock1}. When $0<r<p-1$, take $E$ to be $E_{(n-r)/p}$ and note that $\sym{n-r}$ contains $E$. By Lemma \ref{L;Sdecomposition}, we have $D^\mu{\downarrow_E}\mid D^\lambda{\downarrow_E}$.
The proof of Lemma \ref{L;Pblock1} implicitly says that $D^\mu{\downarrow_E}$ is not generically free. Therefore, We deduce that $D^{\lambda}{\downarrow_E}$ is not generically free by Proposition \ref{P;sjt} (iii). We get $c_{\sym{n}}(D^\lambda)=\frac{n-r}{p}$ by applying Proposition \ref{P;sjt} (vii) to $D^\lambda{\downarrow_E}$.

When $r=p-1$, using Corollary \ref{L;Findpcore} and Theorem \ref{T;James}, we have $S^{\lambda}\sim D^{\lambda}+D^{(n)}$ if $n\equiv p-1 \pmod {p^{2}}$ and $S^{\lambda}\cong D^{\lambda}$ otherwise. The latter situation is trivial by Lemma \ref{L;YongModules} (iii). When $n\equiv p-1 \pmod {p^{2}}$, by $(2.4)$ and Theorem \ref{T; Lucas}, we have $\dim_\F D^{\lambda}\equiv p-2 \pmod p$ and $c_{\sym{n}}(D^{\lambda})=\frac{n-p+1}{p}$. This completes the proof.
\end{proof}

\begin{lem}\label{L;Base Induction}
Let $2p+2<n$, $r=0$ and $\lambda:=(n-p-1,p+1)$. Then $c_{\sym{n}}(D^{\lambda})=\frac{n}{p}$.
\end{lem}
\begin{proof}
Denote the partitions $(n-p,p)$ and $(n-1,1)$ by $\mu$ and $\nu$ respectively. Let $E$ be $E_{n/p}$ and note that $\lambda$ has $p$-weight $\frac{n}{p}$. Set $u_\alpha$ by a generator set of $E$ and let $A$ be the corresponding algebra $\K\langle u_\alpha\rangle$. By Corollary \ref{L;Findpcore} and Theorem \ref{T;James},
\begin{align*}
S^{\lambda}\sim\begin{cases} D^{\lambda}+D^{\mu}+D^{(n)}, &\text{if}\ n\equiv p\ \pmod {p^2},\\
D^{\lambda}+D^{\mu}+D^{\nu}, &\text{if}\ n\equiv 2p \pmod {p^2},\\
D^{\lambda}+D^{\mu}, & \text{otherwise}.\tag{4.2}
\end{cases}
\end{align*}
Let $i$ be an integer where $0\leq i\leq p-1$. By $(4.1)$-$(4.2)$, $(2.4)$, Theorem \ref{T; Lucas},
\begin{align*}
\dim_\F D^{\lambda}\equiv\begin{cases}
p-4\ \ \ \pmod p, &\text{if}\ n\equiv 2p\pmod {p^2},\\
-2i-2 \pmod p,\ &\text{if}\ n\equiv ip\pmod {p^2},\ i\neq2.
\end{cases}
\end{align*}
So we have $p\nmid\dim_\F D^\lambda$ and $c_{\sym{n}}(D^{\lambda})=\frac{n}{p}$ unless $3<p$ and $n\equiv(p-1)p \pmod {p^2}$.

When $3<p$ and $n\equiv (p-1)p \pmod {p^2}$, suppose that $D^{\lambda}{\downarrow_E}$ is generically free. Then $S^{\lambda}{\downarrow_E}$ and $D^{\mu}{\downarrow_E}$ have same stable generic Jordan type by $(4.2)$ and Proposition \ref{P;sjt} (iv) (c). Following \cite[Lemma 5.8]{Lim1}, we know that $S^{\lambda}{\downarrow_E}$ and $S^{\mu}{\downarrow_E}$ have mutually complementary stable generic Jordan types. By Lemmas \ref{L; YP1} (i), \ref{L;(n-1,1)sjt} and Proposition \ref{P;sjt} (ii), notice that stable generic Jordan types of $Y^{\mu}{\downarrow_E}$ and $S_{\nu}{\downarrow_E}$ are $[1]^{\frac{n}{p}-1}$ and $[p-1]^1$ respectively. We apply Theorem \ref{InsertionRule} to the short exact sequence
$0\rightarrow S_{\nu}{\downarrow_A}\rightarrow Y^{\mu}{\downarrow_A}\rightarrow S_{\mu}{\downarrow_A}\rightarrow 0$ given by Lemma \ref{L; YP1} (ii) and see that stable generic Jordan type of $S_\mu{\downarrow_E}$ is in $\{[1]^{\frac{n}{p}}, [1]^{\frac{n}{p}-2}[2]^1 \}$. The fact and Proposition \ref{P;sjt} (ii) imply that stable generic Jordan type of $S^\mu{\downarrow_E}$ is also in $\{[1]^{\frac{n}{p}}, [1]^{\frac{n}{p}-2}[2]^1 \}$. According to above discussion, stable generic Jordan type of $S^{\lambda}{\downarrow_E}$ lies in the set $C$ defined by $\{[p-1]^{\frac{n}{p}}, [p-2]^1[p-1]^{\frac{n}{p}-2}\}$. We also observe that stable generic Jordan type of $D^\nu{\downarrow_E}$ lies in $\{[p-2]^1, [p-1]^2\}$ by Lemma \ref{L;(n-1,1)sjt} and applying Theorem \ref{InsertionRule} to the short exact sequence
$ 0\rightarrow D^{(n)}{\downarrow_{A}}\rightarrow S^{\nu}{\downarrow_A}\rightarrow D^{\nu}{\downarrow_A}\rightarrow 0.$
Now we apply Theorem \ref{InsertionRule} and Third Isomorphism Theorem to the short exact sequence
$0\rightarrow D^\nu{\downarrow_A}\rightarrow S^\mu{\downarrow_A}\rightarrow D^\mu{\downarrow_A}\rightarrow 0$
given by $(4.1)$ to determine the set of all allowable formal stable generic Jordan types that contains stable generic Jordan type of $D^\mu{\downarrow_E}$. Precisely, all members of the set are listed as follows: $$[1]^{\frac{n}{p}+2},\ [1]^{\frac{n}{p}}[2]^1,\ [1]^{\frac{n}{p}-1}[3]^1,\  [1]^{\frac{n}{p}-2}[4]^1,\ [1]^{\frac{n}{p}-2}[2]^2,\ [1]^{\frac{n}{p}-3}[2]^1[3]^1,\ [1]^{\frac{n}{p}-4}[2]^3.$$
Note that these formal Jordan types are all well defined since $n\equiv (p-1)p \pmod {p^2}$ and $5\leq p$. Name the set by $S$ and notice that $C\cap S=\varnothing$. It is a contradiction since we have known that $S^{\lambda}{\downarrow_E}$ and $D^{\mu}{\downarrow_E}$ have same stable generic Jordan type. So $D^{\lambda}{\downarrow_E}$ is not generically free. We have $c_{\sym{n}}(D^\lambda)=\frac{n}{p}$ by Proposition \ref{P;sjt} (vii).
\end{proof}

We now compute $c_{\sym{n}}(D^{(n-p-s,p+s)})$ when $2p+2s\leq n$. By Lemma \ref{JLW}, the partition $(p+s,p+s)$ has $p$-weight two if and only if $0<s<p-1$.
\begin{lem}\label{L;2p+2s}
Let $\lambda:=(p+s,p+s)$ have $p$-weight $w$. Then $c_{\sym{2p+2s}}(D^{\lambda})=w$.
\end{lem}
\begin{proof}
When $0<s<p-1$, by \cite[Theorem 4.9]{LimTan}, one has $c_{\sym{2p+2s}}(D^{(p+s,p+s)})=2$. When $s=p-1$, let $\mu$ and $\nu$ be the partitions $(3p-2,p)$ and $(3p-1,p-1)$ respectively and note that $w=3$. By Corollary \ref{L;Findpcore} and Theorem \ref{T;James}, we have $S^\lambda\sim D^\lambda+D^\mu$, $S^\mu\sim D^\mu+D^\nu$, and $S^\nu\sim D^\nu+D^{(4p-2)}$. Using the relations, $(2.4)$ and Theorem \ref{Lucas}, one gets $\dim_\F D^\alpha\equiv -8\pmod p$, which implies that $c_{\sym{4p-2}}(D^{(2p-1,2p-1)})=3$. The lemma follows.
\end{proof}

\begin{prop}\label{P;p+s}
Let $2p+2s\leq n$ and $\lambda:=(n-p-s,p+s)$ have $p$-weight $w$. Then $c_{\sym{n}}(D^{\lambda})=w$.
\end{prop}
\begin{proof}
Denote the partitions $(n-p-s,p)$, $(n-s,s)$ by $\mu$, $\nu$ respectively. For any integer $a$ satisfying condition $0\leq a\leq s$, denote the partitions $(n-p-s,p+s-a)$ by $\lambda_a$. Note that $\lambda_0=\lambda$ and $\lambda_s=\mu$. Let $t$ be an integer such that $0\leq t<s$. By Lemma \ref{L;2p+2s}, we assume $2p+2s<n$ and have the following cases.
\begin{enumerate}[\text{Case} 1:]
\item $0<s\leq \frac{r}{2}$
\end{enumerate}
Let $E$ be $E_{(n-r)/p}$ and note that $w=\frac{n-r}{p}$ by Lemma \ref{L:pcore}. Notice that $\sym{n-s}$ contains $E$ as $s<r$. We claim that $s-t-2\not\equiv r-s-1 \pmod p$ and $s-t-2\not\equiv r-s \pmod p$. Note that $0<2s\leq r$ and $0<r-2s+t+1=r-s-(s-t)+1<r<p$. The first relation follows. The other relation can be shown similarly. The claim is proved. The claim and Theorem \ref{T;MBR} show that $D^{\lambda_{t+1}}\mid D^{\lambda_t}{\downarrow_{\sym{n-t-1}}}$. We thus get $D^\mu\mid D^\lambda{\downarrow_{\sym{n-s}}}$ by induction. Note that $n-s\equiv r-s \pmod p$ and $0<r-s<p-1$. We thus get that $D^{\mu}{\downarrow_E}$ is not generically free by the proof of Proposition \ref{P;equal p}. Therefore, $D^{\lambda}{\downarrow_E}$ is not generically free by Proposition \ref{P;sjt} (iii), which implies that $c_{\sym{n}}(D^{\lambda})=\frac{n-r}{p}$ by Proposition \ref{P;sjt} (vii).
\begin{enumerate}[\text{Case} 2:]
\item $s=\frac{r+1}{2}$
\end{enumerate}
Let $E$ be $E_{(n-p-r)/p}$ and note that $w=\frac{n-p-r}{p}$ by Lemma \ref{L:pcore}. Set $u_\alpha$ by a generator set of $E$ and let $A$ be the corresponding algebra $\mathbb{K}\langle u_\alpha \rangle$. Note that $\nu$ is the unique partition of $n$ such that $\lambda\lhd \nu$ and $\kappa_p(\nu)=\kappa_p(\lambda)$ by Lemma \ref{L:pcore}. Therefore, we have $S^{\lambda}\sim D^{\lambda}+D^{\nu}$ if $n\equiv p+r \pmod {p^2}$ and $S^{\lambda}\cong D^{\lambda}$ otherwise by Theorem \ref{T;James}. The situation $S^{\lambda}\cong D^{\lambda}$ is trivial by Lemma \ref{L;YongModules} (iii). For the left situation, let $Y^{\lambda}{\downarrow_E}$ and $Y^{\nu}{\downarrow_E}$ have stable generic Jordan types $[1]^a$ and $[1]^b$ respectively. Note that $S^{\nu}\cong D^{\nu}\cong Y^{\nu}$ and $b<a$ by Lemmas \ref{L; YP2} (i) and \ref{L;p=3,YP2}. Suppose $D^{\lambda}{\downarrow_E}$ is generically free. Then both $S^{\lambda}{\downarrow_E}$ and $D^{\nu}{\downarrow_E}$ have stable generic Jordan type $[1]^b$ by the relation $S^{\lambda}\sim D^{\lambda}+D^{\nu}$ and Proposition \ref{P;sjt} (iv) (c). We apply Theorem \ref{InsertionRule} and Third Isomorphism Theorem to the short exact sequence
$0\rightarrow S^{\lambda}{\downarrow_A}\rightarrow Y^{\lambda}{\downarrow_A}\rightarrow S^{\nu}{\downarrow_A}\rightarrow 0$ given in Lemma \ref{L; YP2} (ii)
and get that allowable stable generic Jordan type of $S^\nu{\downarrow_E}$ has at least $a-b$ Jordan blocks of dimension one. It implies that $a-b\leq b$ as $S^\nu{\downarrow_E}$ has stable generic Jordan type $[1]^b$. The inequality contradicts with Lemma \ref{L;p=3,YP2}. So $D^{\lambda}{\downarrow_E}$ is not generically free, which shows that $c_{\sym{n}}(D^{\lambda})=\frac{n-p-r}{p}$ by Proposition \ref{P;sjt} (vii).
\begin{enumerate}[\text{Case} 3:]
\item $\frac{r+1}{2}<s\leq r+1$
\end{enumerate}
Let $E$ be $E_{(n-r)/p}$ and note that $\sym{n-r}$ contains $E$. Observe that $w=\frac{n-r}{p}$ by Lemma \ref{L:pcore}. When $r=0$ and $s=1$, we have
$c_{\sym{n}}(D^{(n-p-1,p+1)})=\frac{n}{p}$ by Lemma \ref{L;Base Induction}. For the case $0<r$, when $\frac{r+2}{2}<s\leq r+1$, let $b:=2s-r-2$ and note that $b<s$. Let $\eta$, $\zeta$, $\tau$ be the partitions $(n-p-s-1, p+r-s+2)$, $(n-p-s-1,p+r-s+1)$, $(n-p-r-1,p+1)$ respectively. We claim that $s-t-2\not\equiv r-s-1 \pmod p$ and $s-t-2\not\equiv r-s \pmod p$ if $0\leq t<b$. Indeed, we have $0<2s-r-2-t+1=2s-r-t-1\leq s-t<p$ when $0\leq t<b$. The first relation is obtained. The proof of the other relation is similar. The claim is shown. By the claim and Theorem \ref{T;MBR}, we have $D^{\lambda_{t+1}}\mid D^{\lambda_t}{\downarrow_{\sym{n-t-1}}}$ when $0\leq t<b$. We thus get $D^{\lambda_b}\mid D^{\lambda}{\downarrow_{\sym{n-b}}}$ by induction and deduce $D^{\tau}\mid D^{\lambda}{\downarrow_{\sym{n-r}}}$ if $s=r+1$. When $s<r+1$, note that $s-b-2\not\equiv r-s-1\pmod p$ while $s-b-2\equiv r-s \pmod p$. Then we have $D^{\eta}\mid D^{\lambda_b}{\downarrow_{\sym{n-b-1}}}$ by Theorem \ref{T;MBR}. Now $s-b-2\not\equiv r-s-2 \pmod p$ and $s-b-2\not\equiv r-s-1 \pmod p$. Hence $D^{\zeta}\mid D^{\eta}{\downarrow_{\sym{n-b-2}}}$ by Theorem \ref{T;MBR}. This pattern will repeatedly occur $r-s$ times when $2(r-s)$ nodes are removed from $\zeta$ by Theorem \ref{T;MBR}. We therefore get $D^{\tau}\mid D^{\lambda}{\downarrow_{\sym{n-r}}}$. When $s=\frac{r+2}{2}$, using Theorem \ref{T;MBR}, we can develop a similar pattern where $r$ nodes are removed from $\lambda$. This pattern will repeatedly occur $\frac{r}{2}$ times. We use the pattern to deduce same result that $D^\tau\mid D^{\lambda}{\downarrow_{\sym{n-r}}}$. The proof of Lemma \ref{L;Base Induction} implicitly says that $D^{\tau}{\downarrow_E}$ is not generically free. Then $D^{\lambda}{\downarrow_E}$ is not generically free by Proposition \ref{P;sjt} (iii). We thus get $c_{\sym{n}}(D^{\lambda})=\frac{n-r}{p}$ by Proposition \ref{P;sjt} (viii).
\begin{enumerate}[\text{Case} 4:]
\item $r+1<s\leq\frac{p+r}{2}$
\end{enumerate} Let $E$ be $E_{(n-p-r)/p}$ and note that $\sym{n-s}$ contains $E$ since $s\leq \frac{p+r}{2}$. Notice that $w=\frac{n-p-r}{p}$ by Lemma \ref{L:pcore}. As in Case $1$, a similar proof shows that $s-t-2\not\equiv r-s-1 \pmod p$ and $s-t-2\not\equiv r-s \pmod p$, which implies that $D^{\mu}\mid D^\lambda{\downarrow_{\sym{n-s}}}$ by Theorem \ref{T;MBR} and induction. Note that $n-s\equiv p+r-s\pmod p$ and $0<p+r-s<p-1$. Same reason given in Case $1$ says that $D^{\lambda}{\downarrow_E}$ is not generically free and $c_{\sym{n}}(D^{\lambda})=\frac{n-p-r}{p}$.
\begin{enumerate}[\text{Case} 5:]
\item $s=\frac{p+r+1}{2}$
\end{enumerate}
Let $E$ be $E_{(n-2p-r)/p}$ and note that $w=\frac{n-2p-r}{p}$ by Lemma \ref{L:pcore}. Notice that $\nu$ is the unique partition of $n$ such that $\lambda\lhd\nu$ and $\kappa_p(\nu)=\kappa_p(\lambda)$ by Lemma \ref{L:pcore}. We have $S^{\lambda}\sim D^{\lambda}+D^{\nu}$ if $n\equiv 2p+r \pmod {p^2}$ and $S^{\lambda}\cong D^{\lambda}$ otherwise by Theorem \ref{T;James}. One can follow the proof of Case $2$ to write a similar proof to show that $D^{\lambda}{\downarrow_E}$ is not generically free and $c_{\sym{n}}(D^{\lambda})=\frac{n-2p-r}{p}$. In the proof, Lemmas \ref{L;YP3} and \ref{L;p=3,YP3} play the roles of Lemmas \ref{L; YP2} and \ref{L;p=3,YP2} respectively.
\begin{enumerate}[\text{Case} 6:]
\item $\frac{p+r+1}{2}<s<p$
\end{enumerate}
Let $E$ be $E_{(n-p-r)/p}$ and note that $\sym{n-p-r}$ contains $E$. Observe that $w=\frac{n-p-r}{p}$ by Lemma \ref{L:pcore}. Let $b:=2s-p-r-2$ and note that $b<s$. Let $\eta$, $\zeta$ and $\tau$ be the partitions $(n-p-s-1, 2p+r-s+2)$, $(n-p-s-1,2p+r-s+1)$ and $(n-2p-r-1,p+1)$ respectively. When $\frac{p+r+2}{2}<s$, as in Case $3$, we can show $s-t-2\not\equiv r-s-1 \pmod p$ and $s-t-2\not\equiv r-s \pmod p$ if $0\leq t<b$ and thus deduce $D^{\lambda_b}\mid D^{\lambda}{\downarrow_{\sym{n-b}}}$. By similar reasons given in Case $3$, we have $D^{\eta}\mid D^{\lambda_b}{\downarrow_{\sym{n-b-1}}}$ and $D^{\zeta}\mid D^{\eta}{\downarrow_{\sym{n-b-2}}}$. The pattern will repeatedly occur $p+r-s$ times by Theorem \ref{T;MBR} when $2(p+r-s)$ nodes are removed from $\zeta$. Finally, we get $D^{\tau}\mid D^{\lambda}{\downarrow_{\sym{n-p-r}}}$. When $s=\frac{p+r+2}{2}$, we can develop a similar pattern as above by Theorem \ref{T;MBR} when $p+r$ nodes are removed from $\lambda$. This pattern will repeatedly occur $\frac{p+r}{2}$ times. We can use the pattern to deduce that $D^\tau\mid D^{\lambda}{\downarrow_{\sym{n-p-r}}}$. By same reason provided in Case $3$, we have $D^{\lambda}{\downarrow_E}$ is not generically free and $c_{\sym{n}}(D^{\lambda})=\frac{n-p-r}{p}$.
\end{proof}

\begin{lem}\label{L;Dimension(n-2p,2p)}
Let $3<p$, $4p\leq n$ and $\lambda:=(n-2p,2p)$. Then $c_{\sym{n}}(D^{\lambda})=\frac{n}{p}$ if $r=0$.
\end{lem}
\begin{proof}
Denote the partitions $(n-p-1,p+1)$, $(n-p,p)$ and $(n-1,1)$ by $\mu$, $\nu$ and $\eta$ respectively. Following Corollary \ref{L;Findpcore} and Theorem \ref{T;James},
\begin{align*}
S^{\lambda}\sim \begin{cases} D^{\lambda}+D^{\mu}+D^{\eta}+D^{(n)}, &\text{if}\ n\equiv 2p \pmod {p^2},\\
D^{\lambda}+D^{\mu}+D^{\nu}, &\text{if}\ n\equiv3p\pmod {p^2},\\
D^{\lambda}+D^{\mu}, & \text{otherwise}.\tag{4.3}\end{cases}
\end{align*}
According to $(4.1)$-$(4.3)$, $(2.4)$ and Theorem \ref{T; Lucas},
\begin{align*}
\dim_\F D^{\lambda}\equiv\begin{cases} 2 \ \ \ \pmod p, &\text{if}\ n\equiv 0 \ \ \ \ \ \ \ \ \ \pmod {p^2},\\
4\ \ \ \pmod p, &\text{if}\ n\equiv p \ \ \ \ \ \ \ \ \ \pmod {p^2},\\
6 \ \ \ \pmod p, & \text{if}\ n\equiv 2p\ \text{or}\ 3p \pmod {p^2},\\
16 \ \pmod p, &\text{if}\ n\equiv 4p \ \ \ \ \ \ \ \ \pmod {p^2}.
\end{cases}
\end{align*}
Let $i$ be an integer. When $n\equiv ip\pmod {p^2}$ and $0\leq i\leq 4$, we conclude that $p\nmid \dim_\F D^{\lambda}$ and $c_{\sym{n}}(D^{\lambda})=\frac{n}{p}$. For the left situations where $n\equiv ip \pmod {p^2}$ and $4<i\leq p-1$, let $\zeta$ and $\tau$ be the partitions  $(n-2p-1,2p+1)$, $(n-2p-1,2p)$ respectively and assign $E$ to be $E_{n/p}$. We have $S^{\lambda}\sim D^{\lambda}+D^{\mu}$, $S^{\zeta}\sim D^{\zeta}+D^{\lambda}$ and $S^{\tau}\cong D^{\tau}$ by Corollary \ref{L;Findpcore} and Theorem \ref{T;James}. The proof of Lemma \ref{L;Base Induction} implicitly says that $D^{\mu}{\downarrow_E}$ is not generically free. According to these facts, one can mimic the proof of the situation $n\equiv (p-2)p \pmod {p^2}$ of Lemma \ref{L;Pblock1} to show that $D^{\lambda}{\downarrow_E}$ is not generically free. The proof of the lemma is now complete by applying Proposition \ref{P;sjt} (vii) to $D^{\lambda}{\downarrow_E}$.
\end{proof}

\begin{prop}\label{P;equal 2p}
Let $3<p$, $4p\leq n$ and $\lambda:=(n-2p,2p)$. Then $c_{\sym{n}}(D^{\lambda})=\frac{n-r}{p}$.
\end{prop}
\begin{proof}
Denote the partitions $(n-2p-r,2p)$ and $(n-p,p)$ by $\mu$ and $\nu$ respectively. We only need to deal with the case $r\neq0$ by Lemma \ref{L;Dimension(n-2p,2p)}. When $1<r<p-1$,
Take $E$ to be $E_{(n-r)/p}$ and note that $\sym{n-r}$ contains $E$. By Lemma \ref{L;Sdecomposition}, we have $D^\mu{\downarrow_E}\mid D^\lambda{\downarrow_E}$. The proof of Lemma \ref{L;Dimension(n-2p,2p)} implicitly gives that $D^\mu{\downarrow_E}$ is not generically free. Therefore, we obtain that $D^{\lambda}{\downarrow_E}$ is not generically free by Proposition \ref{P;sjt} (iii). We get $c_{\sym{n}}(D^\lambda)=\frac{n-r}{p}$ by applying Proposition \ref{P;sjt} (vii) to $D^\lambda{\downarrow_E}$.

When $r=p-1$, using Corollary \ref{L;Findpcore} and Theorem \ref{T;James},
\begin{align*}
S^{\lambda}\sim \begin{cases} D^{\lambda}+D^{(n)}, & \text{if}\ n\equiv 2p-1 \pmod {p^2},\\
D^{\lambda}+D^{\nu}, & \text{if}\ n\equiv 3p-1 \pmod {p^2},\\
D^{\lambda}, & \text{otherwise}.\end{cases}
\end{align*}
The situation $S^{\lambda}\cong D^{\lambda}$ is trivial by Lemma \ref{L;YongModules} (iii). For the left situations, we deduce that $\dim_\F D^{\lambda}\equiv p-2 \pmod p$ by above relations, $(2.4)$ and Theorem \ref{T; Lucas}, which implies $c_{\sym{n}}(D^\lambda)=\frac{n-r}{p}$. The proof is now complete.
\end{proof}

As the complexity of the trivial $\F\sym{n}$-module is clear, we now finish the proof of Theorem \ref{T;C} by collecting Propositions \ref{P;Less p},\ \ref{P;equal p},\ \ref{P;p+s} and \ref{P;equal 2p}.

\end{document}